\documentclass[reqno]{amsart}
\usepackage{rotating,amsmath,amssymb,amsfonts}
\usepackage{psfrag,epsfig,url,color}
\numberwithin{figure}{section}
\numberwithin{table}{section}
\newtheorem{theorem}{Theorem}[section]
\newtheorem{lemma}[theorem]{Lemma}

\newtheorem{conjecture}[theorem]{Conjecture}
\newtheorem{corollary}[theorem]{Corollary}
\theoremstyle{definition}
\newtheorem{definition}[theorem]{Definition}
\newtheorem{example}[theorem]{Example}
\newtheorem{proposition}[theorem]{Proposition}

\newtheorem{remark}[theorem]{Remark}
\numberwithin{equation}{section}

\newcommand{\Ext}{\operatorname{Ext}}

\newcommand{\kqm}{$kQ$-$\mod$ }
\renewcommand{\mod}{\operatorname{mod}}

\def \N{{\mathbb N}}

\def \h{{\mathfrak h}}
\def \Z{{\mathbb Z}}

\def \A{{\mathfrak A}}
\def \[{[ }
\def \]{] }
\def \M{{\mathcal{M}}}

\def\F{\mathcal{F}}

\def\t{\widetilde}

\def\h{\widehat}

\title[Growth rate of cluster algebras]{Growth rate of cluster algebras}

\author{Anna Felikson, Pavel Tumarkin}
\address{Department of Mathematical Sciences, Durham University, Science Laboratories, South Road, Durham, DH1 3LE, UK}
\email{anna.felikson@durham.ac.uk\\
   pavel.tumarkin@durham.ac.uk}
\thanks{Research was partially supported by DFG grant FE-1241/2 (A.F.), grants DMS 0800671 and DMS 1101369 (M.S.), an NSERC Discovery Grant (H.T.), and RFBR grant 11-01-00289-a (P.T.)}

\author{Michael Shapiro}
\address{Department of Mathematics, Michigan State University, East Lansing, MI 48824, USA}
\email{mshapiro@math.msu.edu}

\author{Hugh Thomas}
\address{Department of Mathematics and Statistics, University of New Brunswick, Fredericton, NB,  
E3B 5A3, Canada}
\email{hugh@math.unb.ca}

\begin{document}
\maketitle

\begin{abstract}
We complete the computation of growth rate of cluster algebras. In particular, we show that growth of all exceptional non-affine mutation-finite cluster algebras is exponential.
\end{abstract}

\setcounter{tocdepth}{1}
\tableofcontents
\section{Introduction}
\label{intro}
This is the fourth paper in the series started in~\cite{FST1,FST2,FST3}.

Cluster algebras were introduced by Fomin and Zelevinsky in the series of papers~\cite{FZ1}, ~\cite{FZ2},~\cite{BFZ3},~\cite{FZ4}. Up to isomorphism, each cluster algebra is defined by a skew-symmetrizable $n\times n$ matrix called its {\it exchange matrix}. Exchange matrices admit {\it mutations} which can be explicitly described. The cluster algebra itself is a commutative algebra with a distinguished set of generators. All the generators are organized into \emph{clusters}. Each cluster contains exactly $n$ generators (\emph{cluster variables}) for a rank $n$ cluster algebra.

Clusters form a nice combinatorial structure. Namely, clusters can be associated with the vertices of $n$-regular tree where the collections of generators in neighboring vertices are connected by relations of an especially simple form called \emph{cluster exchange relations}.
Exchange relations are governed by the corresponding exchange matrix which in its turn undergoes cluster mutations as described above. The combinatorics of the cluster algebra is encoded  by its \emph{exchange graph}, which can be obtained from the $n$-regular tree by identifying vertices with equal clusters (i.e., the clusters containing the same collection of cluster variables).

This paper is devoted to the computation of the growth rate of exchange graphs of cluster algebras.  We say that a cluster algebra is \emph{of exponential growth} if the number of distinct vertices of the exchange graph inside a circle of radius $N$, i.e., that can be reached from an initial vertex in $N$ mutations, grows exponentially in $N$. We say that the growth of a cluster algebra is \emph{polynomial} if this number grows at most polynomially depending on $N$.

In~\cite{FST} Fomin, Shapiro and Thurston computed the growth of \emph{cluster algebras originating from surfaces} (or simply {\it cluster algebras from surfaces} for short). This special class of cluster algebras is characterized by their exchange matrices being signed adjacency matrices of ideal triangulations of marked bordered surfaces. In particular, these matrices are skew-symmetric (we call a cluster algebra \emph{skew-symmetric} if its exchange matrices are skew-symmetric, otherwise we call it \emph{skew-symmetrizable}). Such an algebra has polynomial growth if the corresponding surface is a sphere with at most three holes and marked points in total, and exponential growth otherwise.

Cluster algebras from surfaces have another interesting property: the collections of their exchange matrices (called {\it mutation classes}) are finite. We call such algebras (and exchange matrices) {\it mutation-finite}. It was shown in~\cite{FST1} that signed adjacency matrices of ideal triangulations almost exhaust the class of mutation-finite skew-symmetric matrices, namely, there are only eleven (exceptional) finite mutation classes of matrices of size at least $3\times 3$ not coming from triangulations of surfaces. It was also proved in~\cite{FST1} that skew-symmetric algebras that are not mutation-finite (we call them {\it mutation-infinite}) are of exponential growth.

In~\cite{FST2}, we classify skew-symmetrizable mutation-finite cluster algebras. The geometric meaning of this classification is clarified in~\cite{FST3}: all but seven finite mutation classes of skew-symmetrizable (non-skew-symmetric) matrices can be obtained via signed adjacency matrices of ideal triangulations of orbifolds. In the same paper~\cite{FST3} we show that the exchange graph of every cluster algebra originating from an orbifold is quasi-isometric to an exchange graph of a cluster algebra from a certain surface. In this way we compute the growth rate of all cluster algebras from orbifolds.


In~\cite{FZ2}, Fomin and Zelevinsky classified all finite cluster algebras, i.e., cluster algebras with finitely many clusters. Their ground-breaking result states that any finite cluster algebra corresponds to one of the finite root systems. More precisely, a ``symmetrization'' of some of the exchange matrices in the corresponding mutation class is a Cartan matrix of the corresponding root system. This observation justifies the following terminology. We say that a cluster algebra is {\it of finite} (or {\it affine}) type if a certain sign symmetric version of one of the  exchange matrices in the corresponding mutation class is the Cartan matrix of the root system.

Now we are ready to formulate the main result of the current paper. For simplicity reasons we state our result in terms of {\it diagrams} (see Section~\ref{diagrams}) rather than in terms of matrices.

\begin{theorem}\label{thm-main}
A cluster algebra $\A$ has polynomial growth if one of the following holds
\begin{enumerate}
\item $\A$ has rank $2$ (finite or linear growth);
\item $\A$ is of one of the following types:
\begin{enumerate}
\item finite type $A_n$, $B_n$, $C_n$, $D_n$, $E_6$, $E_7$, $E_8$, $F_4$, or $G_2$, then $\A$ is finite;
\item affine type $\t A_n$, $\t B_n$, $\t C_n$, or $\t D_n$, then $\A$ has linear growth;
\end{enumerate}

\item the mutation class contains one of the following three diagrams shown in Fig.~\ref{growth-diagr}:
\begin{enumerate}
\item diagram $\Gamma(n_1, n_2),\; n_1, n_2 \in \Z_{>0}$, then $\A$ has quadratic growth;
\item diagram $\Delta(n_1, n_2),\; n_1, n_2 \in \Z_{>0}$, then $\A$ has quadratic growth;
\item  diagram $\Gamma(n_1, n_2, n_3),\; n_1, n_2, n_3 \in \Z_{>0}$, then $\A$ has cubic growth;
\end{enumerate}
\item $\A$ is of one of the following exceptional affine types:
\begin{enumerate}
\item $\t E_6$, $\t E_7$, $\t E_8$, then $\A$ is skew-symmetric of linear growth;
\item $\t G_2$, $\t F_4$, then $\A$ is skew-symmetrizable of linear growth.
\end{enumerate}
\end{enumerate}
Otherwise, $\A$ has exponential growth.
\end{theorem}

\begin{figure}[!h]
\begin{center}
\begin{tabular}{rl}
 $\Gamma(n_1,n_2)$&
\psfrag{a1}{\scriptsize $a_1$}
\psfrag{a2}{\scriptsize $a_2$}
\psfrag{an1-}{\scriptsize $a_{n_1-1}$}
\psfrag{an1}{\scriptsize $a_{n_1}$}
\psfrag{b1}{\scriptsize $b_1$}
\psfrag{b2}{\scriptsize $b_2$}
\psfrag{bn2+}{\scriptsize $b_{n_2+1}$}
\psfrag{bn2}{\scriptsize $b_{n_2}$}
\psfrag{b0}{\scriptsize $b_0$}
\psfrag{b0'}{\scriptsize $b_0'$}
\psfrag{c1}{\scriptsize $c_1$}
\psfrag{cn3-}{\scriptsize $c_{n_3-1}$}
\psfrag{cn3}{\scriptsize $c_{n_3}$}
\psfrag{dots}{\scriptsize $\dots$}
\psfrag{4}{\small $4$}
\raisebox{-18pt}{
\epsfig{file=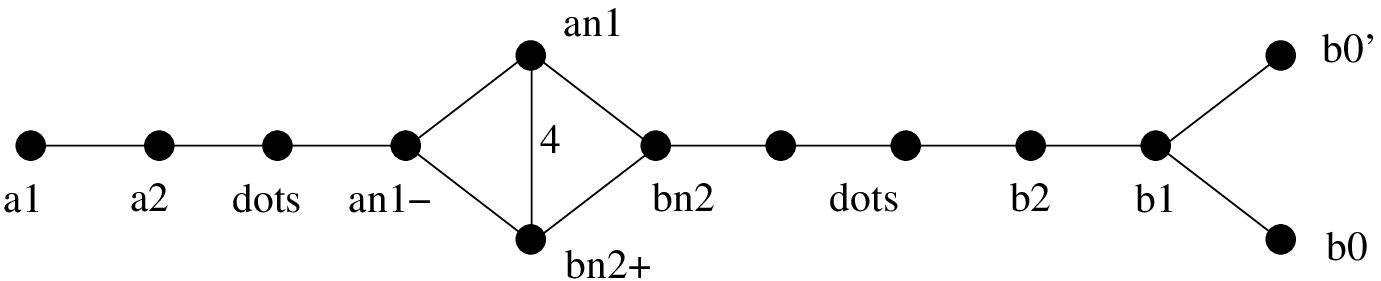,width=0.716\linewidth}
}
\\
 $\Gamma(n_1,n_2,n_3)$&
\psfrag{a1}{\scriptsize $a_1$}
\psfrag{a2}{\scriptsize $a_2$}
\psfrag{an1-}{\scriptsize $a_{n_1-1}$}
\psfrag{an1}{\scriptsize $a_{n_1}$}
\psfrag{b1}{\scriptsize $b_1$}
\psfrag{b2}{\scriptsize $b_2$}
\psfrag{bn2+}{\scriptsize $b_{n_2+1}$}
\psfrag{bn2++}{\scriptsize $b_{n_2+2}$}
\psfrag{bn2}{\scriptsize $b_{n_2}$}
\psfrag{b0}{\scriptsize $b_0$}
\psfrag{b0'}{\scriptsize $b_0'$}
\psfrag{c1}{\scriptsize $c_1$}
\psfrag{cn3-}{\scriptsize $c_{n_3-1}$}
\psfrag{cn3}{\scriptsize $c_{n_3}$}
\psfrag{dots}{\scriptsize $\dots$}
\raisebox{-18pt}{
\epsfig{file=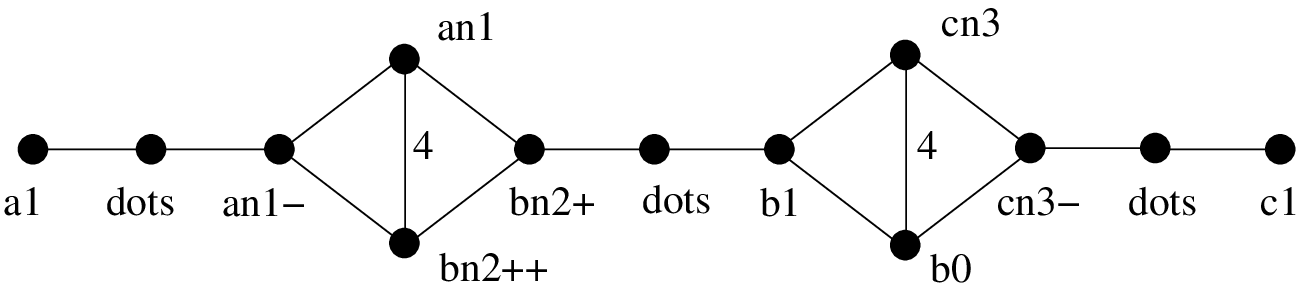,width=0.68\linewidth}
}\\
 $\Delta(n_1,n_2)$&
\psfrag{a1}{\scriptsize $a_1$}
\psfrag{a2}{\scriptsize $a_2$}
\psfrag{an1-}{\scriptsize $a_{n_1-1}$}
\psfrag{an1}{\scriptsize $a_{n_1}$}
\psfrag{b1}{\scriptsize $b_1$}
\psfrag{b2}{\scriptsize $b_2$}
\psfrag{bn2+}{\scriptsize $b_{n_2+1}$}
\psfrag{bn2}{\scriptsize $b_{n_2}$}
\psfrag{b0}{\scriptsize $b_0$}
\psfrag{b0'}{\scriptsize $b_0'$}
\psfrag{c1}{\scriptsize $c_1$}
\psfrag{cn3-}{\scriptsize $c_{n_3-1}$}
\psfrag{cn3}{\scriptsize $c_{n_3}$}
\psfrag{dots}{\scriptsize $\dots$}
\psfrag{2}{\small $2$}
\raisebox{-18pt}{
\epsfig{file=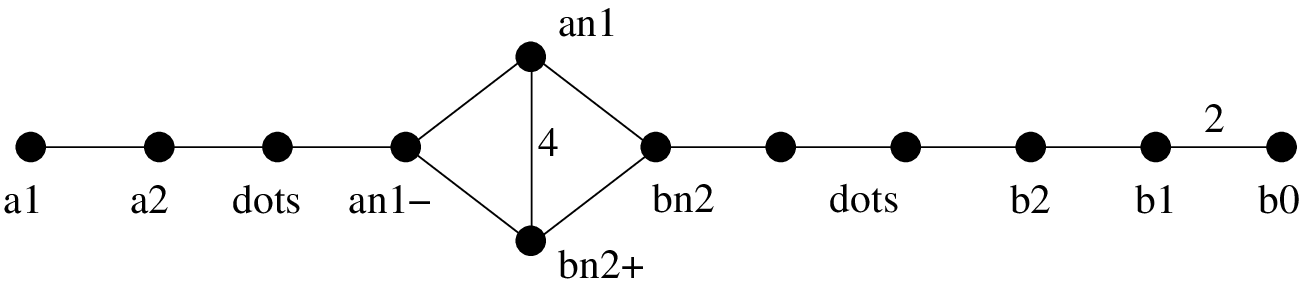,width=0.68\linewidth}
}\end{tabular}
\caption{Diagrams for the cluster algebras of quadratic and cubic growth. All triangles are oriented. Orientations of the remaining
edges are of no importance.}
\label{growth-diagr}
\end{center}
\end{figure}

\begin{remark}
Another independent proof of exponential growth for tubular cluster algebras (namely, $D_4^{(1,1)}, E_6^{(1,1)}, E_7^{(1,1)}, E_8^{(1,1)}$) is obtained recently in \cite{BGJ}.
\end{remark}

\begin{remark}
\label{non-connected}
In the paper we consider cluster algebras with connected diagrams only. Nathan Reading mentioned to us that growth of cluster algebras with non-connected diagrams can also be derived from Theorem~\ref{thm-main}. Indeed, the number of pairs of vertices in a disjoint union of several rooted graphs at total distance $N$ from the roots is a convolution of the respective functions for the connected components. In particular, this implies that the growth of a cluster algebra is polynomial if and only if for every connected component of its diagram the growth of corresponding cluster algebra is polynomial; the only cluster algebras of linear growth are affine ones with connected diagrams.     

\end{remark}

The plan of the proof is as follows. Note first that the case (1) of rank two cluster algebras is evident: the exchange graph is either a finite cycle (finite case: $A_2, B_2, G_2$) or it is an infinite path implying linear growth rate of the cluster algebra. The case (2a) is also clear.

As the next step we mention (Lemma~\ref{inf-exp}) that any mutation-infinite cluster algebra has exponential growth (see also ~\cite{FST1}). The latter implies that it remains only to determine the growth rate of cluster algebras of finite mutation type.

We collect all already known results on the growth of cluster algebras from surfaces and orbifolds in Section~\ref{sec:all}. This covers cases (2b) and (3).
The polynomial growth of skew-symmetric affine exceptional types (case 4a) is proved using the categorification approach to cluster algebras (see Section~\ref{sec:Hugh}). The case (4b) follows from (4a) via the {\it unfolding} construction recalled in Section~\ref{unf}.
Thus, we are left to prove exponential growth of all the remaining exceptional mutation-finite cluster algebras.

The \emph{mapping class group} of a cluster algebra consists of sequences of mutations that preserve the initial exchange matrix. All nontrivial elements of the mapping class group change the cluster, in particular, different elements of the mapping class group produce different clusters from the initial one. Hence, the exponential growth of the mapping class group implies the exponential growth of the cluster algebra.

 To prove exponential growth of remaining exceptional cases we utilize the famous ``ping-pong lemma'' used in the proof of Tits alternative, that allows us to find a free group with two generators as a subgroup of the mapping class group of a corresponding cluster algebra.  This provides an exponential growth of the mapping class group which, in its turn, implies exponential growth of the corresponding cluster algebra. 

To apply the ping-pong lemma we consider mutations of $g$-vectors (see Section~\ref{sec:g-vectors}). The strategy consists of finding two elements of the mapping class group such that  their actions on the space of $g$-vectors satisfy conditions of the ping-pong lemma. The proof is accomplished by the detailed case-by-case analysis of $g$-vector mutations for an appropriate pair of elements of the mapping class group in each exceptional case.

Note that up to this moment we work in coefficient-free settings. In Section~\ref{coeff} we show that growth of cluster algebras does not depend on the coefficients, so the main theorem holds in full generality.

{\bf Acknowledgments}. It is a pleasure to thank the Hausdorff Research Institute for Mathematics whose hospitality the second author enjoyed in the summer of 2011, and the Banff Research Center for hosting a workshop on cluster algebras in September 2011 where the final version of the paper was prepared. We are grateful to L.~Chekhov, V.~Fock, S.~Fomin, M.~Gekhtman, Chr.~Geiss, N.~Ivanov, B.~Keller, B.~Leclerc, and A.~Vainshtein for stimulating discussions. We thank the anonymous referee for valuable comments and suggestions. We also thank N.~Reading for Remark~\ref{non-connected}. 

\section{Exchange matrices and diagrams}
\label{diagrams}

\subsection{Diagram of a skew-symmetrizable matrix}
Following~\cite{FZ2}, we encode an $n\times n$ skew-symmetrizable
integer matrix $B$ by a finite simplicial $1$-complex $S$ with
oriented weighted edges called a {\it diagram}. The weights of a
diagram are positive integers.

Vertices of $S$ are labeled by $[1,\dots,n]$. If $b_{ij}>0$, we join
vertices $i$ and $j$ by an  edge directed from $i$ to $j$ and assign
to this edge weight $-b_{ij}b_{ji}$. Not every diagram corresponds
to a skew-symmetrizable integer matrix: given a diagram $S$, there
exists a skew-symmetrizable integer matrix $B$ with diagram $S$ if
and only if a product of weights along any chordless cycle of $S$ is
a perfect square.

Distinct matrices may have the same diagram. At the same time, it is
easy to see that only finitely many matrices may correspond to the
same diagram. All weights of a diagram of a skew-symmetric matrix
are perfect squares. Conversely, if all weights of a diagram $S$ are
perfect squares, then there is a skew-symmetric matrix $B$ with
diagram $S$.

As it is shown in~\cite{FZ2}, mutations of exchange matrices induce
{\it mutations of diagrams}. If $S$ is the diagram corresponding to
matrix $B$, and $B'$ is a mutation of $B$ in direction $k$, then we
call the diagram $S'$ associated to $B'$ a {\it mutation of $S$ in
direction $k$} and denote it by $\mu_k(S)$. A mutation in direction
$k$ changes weights of diagram in the way described in Fig.~\ref{quivermut} (see e.g.~\cite{FZ2}).

\begin{figure}[!h]
\begin{center}
\psfrag{a}{\small $a$}
\psfrag{b}{\small $b$}
\psfrag{c}{\small $c$}
\psfrag{d}{\small $d$}
\psfrag{k}{\small $k$}
\psfrag{mu}{\small $\mu_k$}
\epsfig{file=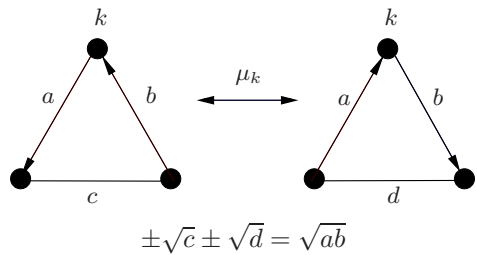,width=0.4\linewidth}\\
\medskip
$\pm\sqrt{c}\pm\sqrt{d}=\sqrt{ab}$
\caption{Mutations of diagrams. The sign before $\sqrt{c}$ (resp., $\sqrt{d}$) is positive if the three vertices form an oriented cycle, and negative otherwise. Either $c$ or $d$ may vanish. If $ab$ is equal to zero then neither the value of $c$ nor the orientation of the corresponding edge changes.}
\label{quivermut}
\end{center}
\end{figure}

For a given diagram, the notion of {\it mutation class} is
well-defined. We call a diagram {\it mutation-finite} if its
mutation class is finite.

The following criterion for a diagram to be mutation-finite is
well-known (see e.g. \cite[Theorem 2.8]{FST2}).

\begin{lemma}
\label{less3} A diagram $S$ of order at least $3$ is mutation-finite
if and only if any diagram in the mutation class of $S$ contains no
edges of weight greater than $4$.
\end{lemma}

\subsection{Unfolding of a skew-symmetrizable matrix}
\label{unf}

In this section, we recall the notion of {\it unfolding} of a skew-symmetrizable matrix.

Let $B$ be an indecomposable $n\times n$ skew-symmetrizable integer matrix, and let $BD$ be a skew-symmetric matrix, where $D=(d_{i})$ is diagonal integer matrix with positive diagonal entries. Notice that for any matrix $\mu_i(B)$ the matrix $\mu_i(B)D$ will be skew-symmetric.

We use the following definition of unfolding (communicated to us by A.~Zelevinsky) (see~\cite{FST2} and~\cite{FST3} for details).

Suppose that we have chosen disjoint index sets $E_1,\dots, E_n$ with $|E_i| =d_i$. Denote $m=\sum\limits_{i=1}^n d_i$.
Suppose also that we choose a skew-symmetric integer matrix $C$ of size $m\times m$ with rows and columns indexed by the union of all $E_i$, such that

(1) the sum of entries in each column of each $E_i \times E_j$ block of $C$ equals $b_{ij}$;

(2) if $b_{ij} \geq 0$ then the $E_i \times E_j$ block of $C$ has all entries non-negative.

Define a {\it composite mutation} $\h\mu_i = \prod_{\hat\imath \in E_i} \mu_{\hat\imath}$ on $C$. This mutation is well-defined, since all the mutations  $\mu_{\hat\imath}$, $\hat\imath\in E_i$, for given $i$ commute.

We say that $C$ is an {\it unfolding} for $B$ if $C$ satisfies assertions $(1)$ and $(2)$ above, and for any sequence of iterated mutations $\mu_{k_1}\dots\mu_{k_m}(B)$ the matrix $C'=\h\mu_{k_1}\dots\h\mu_{k_m}(C)$ satisfies assertions $(1)$ and $(2)$ with respect to $B'=\mu_{k_1}\dots\mu_{k_m}(B)$.

\section{Block decompositions of diagrams}
\label{blockdecomp}

In~\cite{FST}, Fomin, Shapiro and Thurston gave a combinatorial description of diagrams of signed adjacency matrices of ideal triangulations. Namely, such diagrams are {\it block-decomposable}, i.e. they are exactly those that can be glued from diagrams shown in Fig.~\ref{bloki} (called {\it blocks}) in the following way.

Call vertices marked in white {\it outlets}. A connected diagram $S$ is called {\it
block-decomposable} if it can be obtained from a collection of
blocks by identifying outlets of different blocks along some partial
matching (matching of outlets of the same block is not allowed),
where two edges with the same endpoints and opposite directions
cancel out, and two edges with the same endpoints and the same
directions form an edge of weight $4$. A non-connected diagram $S$
is called  block-decomposable either if $S$ satisfies the definition
above, or if $S$ is a disjoint union of several diagrams satisfying the definition above.
If $S$ is not block-decomposable then we call $S$ {\it non-decomposable}.

\begin{figure}[!h]
\begin{center}
\psfrag{1}{${\rm{I}}$} \psfrag{2}{${\rm{II}}$}
\psfrag{3a}{${\rm{IIIa}}$} \psfrag{3b}{${\rm{IIIb}}$}
\psfrag{4}{${\rm{IV}}$} \psfrag{5}{${\rm{V}}$}
\epsfig{file=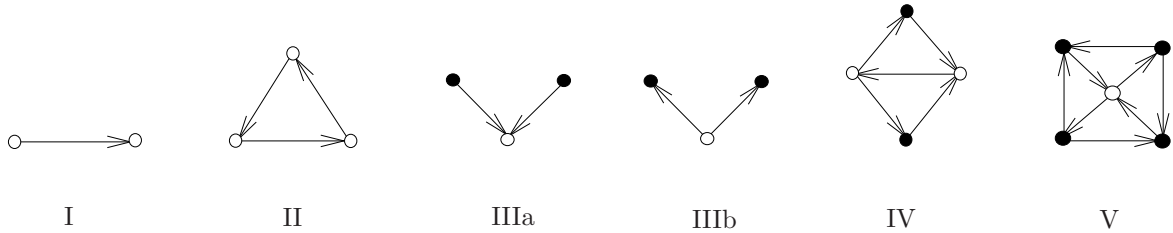,width=0.999\linewidth}
\caption{Blocks. Outlets are colored in white.}
\label{bloki}
\end{center}
\end{figure}

As it was mentioned above, block-decomposable diagrams are in one-to-one correspondence with
adjacency matrices of arcs of ideal (tagged) triangulations of
bordered two-dimensional surfaces with marked points
(see~\cite[Section~13]{FST} for the detailed explanations).
Mutations of block-decomposable diagrams correspond to flips of
triangulations. In particular, this implies that mutation class of
any block-decomposable diagram is finite.

\medskip

It was shown in~\cite{FST2,FST3} that diagrams of signed adjacency matrices of arcs of ideal triangulations of orbifolds can be described in a similar way. For this, we need to introduce new \emph{s-blocks} shown in Fig.~\ref{s-bloki}.

\begin{figure}[!h]
\begin{center}
\psfrag{3at}{${\rm{\t{III}a}}$}
\psfrag{3bt}{${\rm{\t{III}b}}$}
\psfrag{4t}{$\t{\rm{IV}}$}
\psfrag{51t}{$\t{\rm{V}}_1$}
\psfrag{52t}{$\t{\rm{V}}_{2}$}
\psfrag{512t}{$\t{\rm{V}}_{12}$}
\psfrag{6t}{$\t{\rm{VI}}$}
\psfrag{2-}{\scriptsize $2$}
\psfrag{4}{\scriptsize $4$}
\epsfig{file=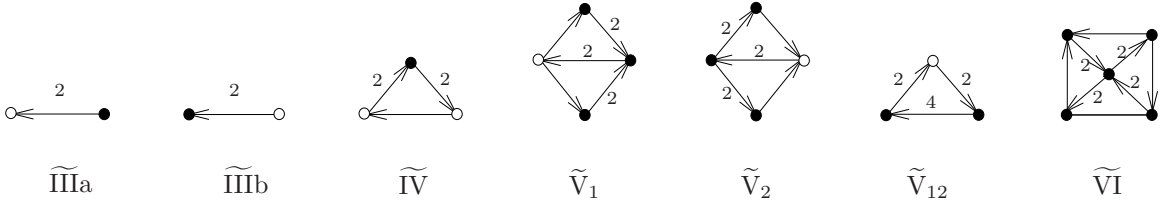,width=0.99\linewidth}
\caption{s-blocks. Outlets are colored in white.}
\label{s-bloki}
\end{center}
\end{figure}

We keep the idea of gluing. A diagram is {\it s-decomposable} if it can be
glued from blocks and s-blocks. We keep the term ``block-decomposable'' for s-decomposable diagrams corresponding to skew-symmetric matrices.

Like block-decomposable diagrams, s-decomposable diagrams are in one-to-one correspondence with
adjacency matrices of arcs of ideal (tagged) triangulations of
bordered two-dimensional orbifolds with marked points and orbifold points of degree two
(see~\cite{FST3}).
As above, mutations of s-decomposable diagrams correspond to flips of
triangulations. This implies that mutation class of
any s-decomposable diagram is also finite.

Therefore, s-decomposable diagrams form a large class of finite mutation diagrams (and therefore exchange matrices).
Moreover, in~\cite{FST2} we proved that together with diagrams of rank 2 they provide almost all diagrams of finite mutation type.

More exactly, the following theorems hold.

\begin{theorem}[Theorem 6.1~\cite{FST1}]\label{all}
A connected non-decomposable skew-symmetric mutation-finite diagram of order greater than $2$ is mutation-equivalent
to one of the eleven diagrams $E_6$, $E_7$, $E_8$, $\widetilde E_6$, $\widetilde E_7$,
$\widetilde E_8$, $X_6$, $X_7$, $E_6^{(1,1)}$, $E_7^{(1,1)}$, $E_8^{(1,1)}$ shown in Figure~\ref{allfig}.

\begin{figure}[!h]
\begin{center}
\psfrag{1}{$E_6$}
\psfrag{2}{$E_7$}
\psfrag{3}{$E_8$}
\psfrag{1_}{$\widetilde E_6$}
\psfrag{2_}{$\widetilde E_7$}
\psfrag{3_}{$\widetilde E_8$}
\psfrag{1__}{$E_6^{(1,1)}$}
\psfrag{2__}{$E_7^{(1,1)}$}
\psfrag{3__}{$E_8^{(1,1)}$}
\psfrag{4-}{$X_6$}
\psfrag{5-}{$X_7$}
\psfrag{4}{\scriptsize $4$}
\epsfig{file=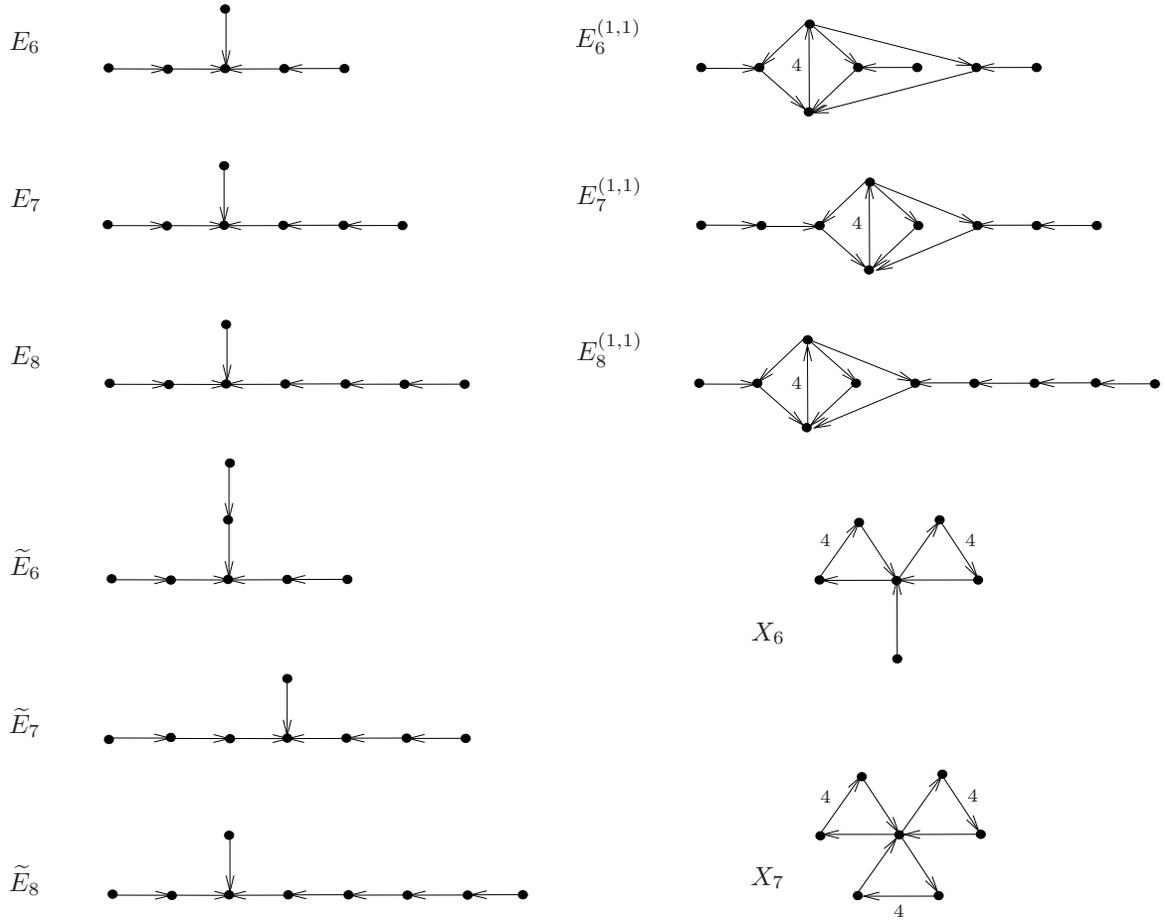,width=.99\linewidth}
\caption{Non-decomposable skew-symmetric mutation-finite diagrams of order at least $3$}
\label{allfig}
\end{center}
\end{figure}

\end{theorem}


\begin{theorem}[Theorem 5.13~\cite{FST2}]\label{allkos}
A connected non-decomposable skew-sym\-metr\-iz\-able diagram, that is not skew-symmetric, has finite mutation class if and only if either it is of order $2$ or its diagram is mutation-equivalent to one of the seven types $\t G_2$, $F_4$, $\t F_4$, $G_2^{(*,+)}$, $G_2^{(*,*)}$, $F_4^{(*,+)}$, $F_4^{(*,*)}$ shown in Fig.~\ref{allfign}.

\begin{figure}[!h]
\begin{center}
\psfrag{2}{\scriptsize $2$}
\psfrag{2-}{\scriptsize $2$}
\psfrag{3}{\scriptsize $3$}
\psfrag{4}{\scriptsize $4$}
\psfrag{G}{$\t G_2$}
\psfrag{F}{$F_4$}
\psfrag{Ft}{$\t F_4$}
\psfrag{W}{$G_2^{(*,*)}$}
\psfrag{V}{$G_2^{(*,+)}$}
\psfrag{Y6}{$F_4^{(*,+)}$}
\psfrag{Z}{$F_4^{(*,*)}$}
\epsfig{file=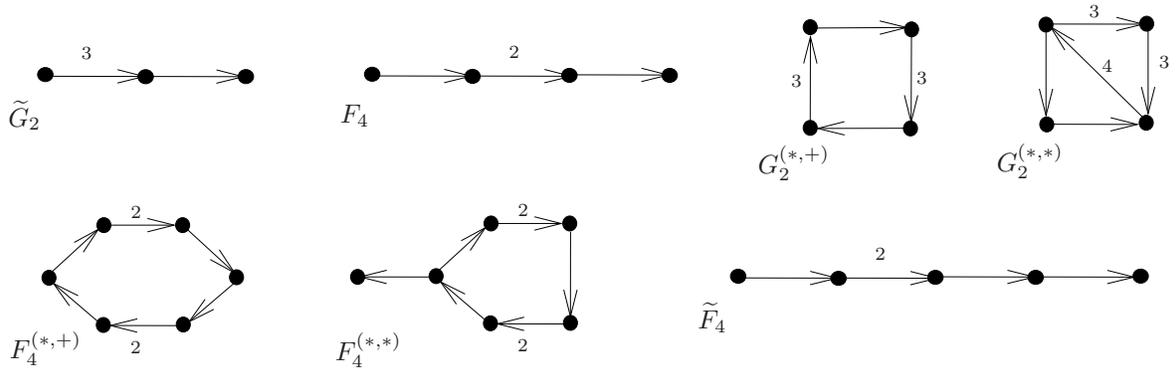,width=1.0\linewidth}
\caption{Non-decomposable mutation-finite non-skew-symmetric diagrams of order at least $3$}
\label{allfign}
\end{center}
\end{figure}

\end{theorem}

\begin{remark}
$\!$In Fig.\!~\ref{allfign}, we have chosen representatives from the mutation classes of non-decomposable diagrams that are slightly different than the ones from~\cite[Theorem~5.13]{FST2}. This is done for simplification of computations in Section~\ref{sec:exp}.

\end{remark}

\section{Growth of non-exceptional cluster algebras}
\label{sec:all}

As was proved in~\cite{FST1}, a mutation-infinite skew-symmetric cluster algebra has exponential growth. Very similar considerations lead to the following lemma (it can also be easily derived from the results of Seven~\cite{Se}).

\begin{lemma}
\label{inf-exp}
Any mutation-infinite skew-symmetrizable cluster algebra has exponential growth.
\end{lemma}

Therefore, we are left to describe the growth of mutation-finite cluster algebras.

According to the results of~\cite{FST1} and~\cite{FST2,FST3}, almost all mutation-finite cluster algebras originate from surfaces or orbifolds.
The growth of cluster algebras from surfaces was computed in~\cite{FST} by investigating mapping class groups of surfaces. In~\cite{FST3}, we compute the growth of cluster algebras from orbifolds by making use of unfoldings and proving quasi-isometry of the corresponding exchange graphs (which is a much stronger statement than needed for growth computation), see~\cite[Section~10]{FST3}. Below we define a mapping class group of a cluster algebra, and then follow~\cite[Section~13]{FST} to present a uniform explanation for both cases. 


Let $\overline n=\{1,2,\dots,n\}$. We denote by $W={\mathbb Z}_2*\dots*{\mathbb Z}_2$ the free product of $n$ copies of ${\mathbb Z}_2$ with $i\in{\overline n}$ being a generator of $i$th copy of ${\mathbb Z}_2$. $W$ is the set of all words without letter repetitions in alphabet $\overline n$. 

A word $w=i_1\,i_2\,\dots\,i_k\in W$ can be interpreted as a sequence $\mu_w$ of mutations of cluster algebra $\A$, namely, $\mu_w=\mu_{i_k}\circ\cdots\circ\mu_{i_2}\circ\mu_{i_1}$. 

\begin{definition} We call a word $w\in W$ \emph{trivial} if $\mu_w(x_i)=x_i$ for any cluster variable $x_i$, $i\in{\overline n}$, of the initial cluster. Trivial words form a subgroup $W_e\subset W$  that we call \emph{the subgroup of trivial transformations}.
\end{definition}

\begin{definition}
A word $w\in W$ is \emph{mutationally trivial} if $\mu_w$ preserves the initial exchange matrix $B$. 
All mutationally trivial words form \emph{a subgroup of mutationally trivial transformations} denoted by 
$W_B\subset W$.
\end{definition}


\begin{lemma}
$W_e\subset W_B$ is a normal subgroup.
\end{lemma}

\begin{proof} Note first that any word $w\in W_e$ preserves exchange matrix $B$ by \cite{GSV} and therefore
$W_e\subset W_B$.
Note also that for all $w\in W_e$, $u\in W_B$ the word $u^{-1} wu$ preserves all initial cluster variables and, hence,
belongs to $W_e$. 
\end{proof}

\begin{definition}
The quotient $\M=W_B/W_e$ is a
\emph{colored mapping class group} of cluster algebra $\A$.
\end{definition}

\begin{example} (cluster algebras of rank $2$)
\begin{enumerate}
\item The group of trivial transformations $W_e$
of the coefficient-free cluster algebra $\A$ of type $A_2$ with the initial exchange 
matrix $B=\begin{pmatrix}
0 & 1 \\
-1 & 0
\end{pmatrix}$ consists of all words $(12)^{5k}$ and $(21)^{5k}$. It
is generated by word $(12)^5$. (Note that  $(21)^5=(12)^{-5}$.)
The group $W_B$ of mutationally trivial transformations is formed by all words $(12)^\ell$ and $(21)^\ell$.
It is generated by the word $(12)$ implying that 
the colored mapping class group $\M=W_B/W_e\simeq {\mathbb Z}_5$.

\item Similarly, for cluster algebras of types $B_2$ and $C_2$ with exchange matrices $B=\begin{pmatrix}
0 & 2 \\
-1 & 0
\end{pmatrix}$ and $B=\begin{pmatrix}
0 & 1 \\
-2 & 0
\end{pmatrix}$ respectively, the colored mapping class group $\M\simeq {\mathbb Z}_6$.
For cluster algebra of type $G_2$ with exchange matrix $B=\begin{pmatrix}
0 & 3 \\
-1 & 0
\end{pmatrix}$ the colored mapping class group $\M\simeq {\mathbb Z}_8$

\item For cluster algebras of non finite type with exchange matrix
$B=\begin{pmatrix}
0 & a \\
-b & 0
\end{pmatrix}$, where $ab\ge 4$, the subgroup $W_e$ of trivial transformations is trivial while $W_B$ is still generated by $(12)$ and the mapping class group 
is an infinite cyclic group, $\M\simeq {\mathbb Z}$.
\end{enumerate}
\end{example} 

\begin{example} Markov cluster algebra. The Markov coefficient-free cluster algebra is a rank 3 cluster algebra with initial exchange matrix  $B=\begin{pmatrix}
\quad0 & \quad2 & -2\\
-2 & \quad0 & \quad2 \\
\quad2 & -2 & \quad0
\end{pmatrix}$. Any simple cluster transformation changes the sign of the exchange matrix. Therefore,
words $(12),\ (13)$, $(21),\ (23)$, $(31)$, $(32)$ generate subgroup $W_B$. Note that $(21)(12)=(13)(31)= (23)(32)=\operatorname{Id}$. Hence, $W_B$ is generated by three mutationally trivial words 
$(12)$, $(13)$, and $(23)$. Recall, that the Markov cluster algebra is a cluster algebra of triangulations of once punctured torus whose mapping class group is known to be isomorphic to $SL_2(\Z)$, and mutationally trivial words represent all the elements of the mapping class group of the torus. The word $(12)$ corresponds to 
$\alpha=\begin{pmatrix}
1 & 0 \\
2 & 1
\end{pmatrix}\in SL_2({\mathbb Z})$, the
word 
$(23)$ to 
$\beta=\begin{pmatrix}
1 & -2 \\
0 & 1
\end{pmatrix}\in SL_2({\mathbb Z})$ and the word $(13)$ to
$\gamma=\begin{pmatrix}
-1 & 2 \\
-2 & 3
\end{pmatrix}\in SL_2({\mathbb Z})$.
Note that $\alpha\beta\gamma^{-1}=-\operatorname{Id}$. It is known (see., e.g.,~\cite{I}), that elements $\alpha$, $\beta$, $-\operatorname{Id}$ generate a principle congruence subgroup $\Gamma(2)$ of $SL_2(\Z)$, consisting of matrices congruent to $\operatorname{Id}$ modulo $2$. The quotient 
$SL_2(\Z)/\Gamma(2)$ is isomorphic to the group $\Sigma_3$ of permutations of three elements. Therefore, the index  $|SL_2(\Z):\Gamma(2)|=6$. 

\end{example}

Denote by $\Sigma_n$ the symmetric group of permutations on $\overline n$.
Any element $\sigma\in\Sigma_n$ acts on any cluster of cluster algebra by a permutation of indices  
$\sigma(x_i)=x_{\sigma(i)}$. This action conjugates the exchange matrix 
by the corresponding permutation matrix $M_\sigma\in SL_n(\Z)$, i.e. $B\mapsto M_\sigma^{-1}BM_\sigma$.
We consider also the action of $\Sigma_n$ on $W$ by a permutation of the letters 
of the alphabet ${\overline n}$.

\begin{definition} 
We call the elements of the set $\widetilde W=W\times \Sigma_n$  \emph{enhanced words}.
\end{definition}

Enhanced word $w\times \sigma\in W\times \Sigma_n$ act on cluster algebra by composition
$\sigma\circ\mu_w$.

\begin{remark}
It is easy to see that $\widetilde W$ is a group. Indeed, the definition of an operation is evident.  The composition $(w_1\times \sigma_1)\circ (w_2\times\sigma_2)$ can be written again as an enhanced word $w_1\sigma_1^{-1}(w_2)\times \sigma_1\sigma_2$. In particular, $(w\times \sigma)^{-1}=\sigma(w^{-1})\times \sigma^{-1}$.
\end{remark}

\begin{definition} An enhanced word $w\times\sigma$ is \emph{trivial} if $(w\times\sigma)x_i=x_i\ \forall i\in{\overline n}$. We denote the subgroup of trivial enhanced words by $\widetilde W_e$. We also denote the subgroup of mutationally trivial enhanced words by $\widetilde W_B$, and \emph{the mapping class group of cluster algebra $\A(B)$} by $\widetilde\M=\widetilde W_B/\widetilde W_e$.
\end{definition}

\begin{example} (cluster algebras of rank 2)
\begin{enumerate}
\item case $A_2$:  The group $\widetilde W_e$ of trivial enhanced transformations 
of the coefficient-free cluster algebra $\A$ of rank $2$ with the initial exchange 
matrix $B=\begin{pmatrix}
0 & 1 \\
-1 & 0
\end{pmatrix}$ consists of enhanced words of the following four types:
$(12121)^{2k}\times \operatorname{Id}$, 
$(21212)^{2k}\times \operatorname{Id}$,
$(12121)^{2k+1}\times\sigma$,
$(21212)^{2k+1}\times \sigma$,
where $\sigma\in\Sigma_2$ is the permutation $(1\leftrightarrow 2)$.
It is generated by element $(12121)\times \sigma$.
The group $\widetilde W_B$ is generated by $(1)\times\sigma$ (note that $(2)\times\sigma=\left((1)\times\sigma\right)^{-1}$, $(21212)\times\sigma=((12121)\times\sigma)^{-1}$, and, finally,  $(12121)\times \sigma=((1)\times\sigma)^5$). Hence, $\widetilde\M\simeq {\mathbb Z}_5$.
\item cases $B_2$, $C_2$, and $G_2$: Similarly, the mapping class groups for $B_2$, $C_2$, and $G_2$ are isomorphic to ${\mathbb Z}_6,\ {\mathbb Z}_6$, and ${\mathbb Z}_8$, respectively.
\item cluster algebra of non finite type: the  mapping class group is ${\mathbb Z}$.
\end{enumerate}
\end{example}

There is a natural embedding $i:W\to \widetilde W$, $i(w)=(w\times\operatorname{Id})$.
Clearly, $i(W_e)\subset \widetilde W_e$ and $i(W_B)\subset \widetilde W_B$.
Therefore, $i$ induces a homomorphism ${\bf i}: \M\to \widetilde\M$.
\begin{lemma}
The map ${\bf i}$ is an embedding.
\end{lemma}
\begin{proof}
Indeed, assume that ${i}(w)\in\widetilde W_e$. Then, $w\times\operatorname{Id}\in \widetilde W_e$.
Hence, $w\times\operatorname{Id}(x_j)=x_j$ implying $\mu_w(x_j)=x_j\,\forall j$. Therefore, $w\in W_e$.
\end{proof}

\begin{remark} Evidently, ${\bf i}(\M)$ is a finite index (normal) subgroup of $\widetilde\M$. Therefore, the growth rate of $\M$ and $\widetilde\M$ is the same.
\end{remark}

\begin{example} (Markov cluster algebra) The mapping class group coincides with the mapping class group
of two-dimensional torus with one puncture which is known to be $SL_2(\Z)$. 
\end{example}


One can note that the mapping class group $\M_S$ of the bordered surface (or bordered orbifold) $S$ is a subgroup of the mapping class group  $\widetilde\M_{\A(S)}$ of the corresponding cluster algebra $\A(S)$.
Indeed, fix a triangulation $T$ of the orbifold. Any element $g$ of the mapping class group of the orbifold can be obtained by some sequence of cluster mutations $s_g(T)=\mu_{i_1}\circ\ldots\circ\mu_{i_k}$ which, however, depends on $T$. At the same time, if the triangulation $T'$ is obtained
from $T$ by a mapping class group action, then $s_g(T')=s_g(T)=\mu_{i_1}\circ\ldots\circ\mu_{i_k}$.
Therefore, if the mapping class group of the surface (orbifold) contains a free group with at least two generators then the cluster algebra has exponential growth.

\begin{remark}  If  the number $m$ of interior marked points on a surface (or orbifold) $S$ is greater than one or $m=1$ and the surface has a nonempty boundary, then the mapping class group $\M_S$ of the surface is a proper normal subgroup of $\widetilde\M_{\A(S)}$ and    
the quotient $\widetilde\M_{\A(S)}/\M_S\simeq {\Z}_2^m$. If $m=0$ or $m=1$ and the surface has no boundary, then $\widetilde\M_{\A(S)}\simeq \M_S$. Indeed, this follows easily from \cite{FST} for surfaces (and \cite{FST3} for orbifolds) where it was shown that for $m>1$
any tagged triangulation can be obtained from any other by a series of flips. Comparing it with the classical result that any two triangulations of the surface are connected by a sequence of flips, and a sequence of flips gives an element of the mapping class group of the surface if and only if the adjacency of the arcs of triangulations are preserved by this sequence, we see that in the first case we can obtain any tagging of marked points, which results in extra $\Z_2$ for every puncture. If $m=1$ and there is no boundary components or if $m=0$, then mutations do not change any tagging. 
\end{remark}

Let us call a {\it feature} of an orbifold (or surface) a hole, a puncture, or an orbifold point.

The above considerations lead to the following theorem.

\begin{theorem}[\cite{FST}, \cite{FST3}]
\label{orbifolds}
Cluster algebras corresponding to orbifolds (or surfaces) of genus $0$ with at most three features have polynomial growth. Cluster algebras corresponding to the other orbifolds (surfaces) grow exponentially.
\end{theorem}

\noindent
Rephrasing this result in terms of diagrams, we obtain the following theorem in~\cite{FST3}.

\begin{theorem}\label{thm:OrbifoldFiniteMutation}
Let $\A$ be a cluster algebra with an $s$-decomposable exchange
matrix $B$. Then $\A$ has polynomial growth if it
corresponds to one of the following diagrams:
\begin{itemize}
\item finite type $A_n$, $B_n$, $C_n$, or $D_n$ (finite);
\item affine type $\t A_n$, $\t B_n$, $\t C_n$, or $\t D_n$  (linear growth);
\item diagram $\Gamma(n_1, n_2) (n_1, n_2 \in \Z_{>0})$ shown in Fig~\ref{growth-diagr}
(quadratic growth);
\item diagram $\Delta(n_1, n_2) (n_1, n_2 \in \Z_{>0})$ shown in Fig.~\ref{growth-diagr}
(quadratic growth);
\item  diagram $\Gamma(n_1, n_2, n_3) (n_1, n_2, n_3 \in \Z_{>0})$ shown in Fig.~\ref{growth-diagr}
(cubic growth).
\end{itemize}
Otherwise $\A$ has exponential growth.
\end{theorem}

\section{Exceptional cluster algebras of exponential growth}\label{sec:exp}
We are left with a short list of exceptional algebras. This section is devoted to the proof of the following theorem.

\begin{theorem}
\label{except_growth}
Cluster algebras with diagrams of types $X_6$, $X_7$, $E_6^{(1,1)}$, $E_7^{(1,1)}$, $E_8^{(1,1)}$,
$G_2^{(*,+)}$, $G_2^{(*,*)}$, $F_4^{(*,+)}$, and $F_4^{(*,*)}$ all have exponential growth.
\end{theorem}

The remaining algebras (of affine types~$\t G_2$, $\t F_4$, $\t E_6$, $\t E_7$, $\t E_8$) are treated in the next section.

\subsection{Ping-pong lemma}\label{sec:ping-pong}
We consider subgroups of the mapping class group $G$ of the corresponding cluster algebra, or the fundamental group of the groupoid of cluster mutations. The elements of the mapping class group are formed by sequences of mutations preserving the chosen initial diagram.


For each exceptional cluster algebra we will find a free subgroup with at least two generators of the mapping class group of the corresponding cluster algebra. Since  the free group with two generators grows exponentially, this implies
an exponential growth of the mapping class group. Different elements of the mapping class group produce different clusters from the initial one, so exponential growth of the mapping class group implies exponential growth of the cluster algebra.

The proof is based on a case-by-case study of the cluster algebras in
question. The main tool is the famous \emph{ping-pong lemma}.



 The ping-pong lemma was a key tool used by Jacques Tits in his 1972 paper~\cite{T} containing the proof of Tits alternative.
Modern versions of the ping-pong lemma can be found in many books, e.g.~\cite{LS}
and others.
We will use the following modification of classical ping-pong lemma~\cite{OS}.

\begin{lemma}\label{lem:ping-pong}
Let $G$ be a group acting on a set $X$ and let $H_1, H_2,\ldots, H_k$
be nontrivial subgroups of $G$ where $k\ge 2$, such that at least one of these subgroups has order greater than $2$. Suppose there exist disjoint nonempty subsets $X_1, X_2,\ldots, X_k$ of $X$ such that the following holds:

For any $i\ne j$ and for any $h\in H_i$, $h\ne 1$ we have $h(X_j)\subset X_i$.

\noindent
Then
 $\langle H_1,\dots, H_k\rangle=H_1\ast\dots \ast H_k$.
\end{lemma}

\begin{corollary}\label{cor:ping-pong}
With the assumptions of Lemma~\ref{lem:ping-pong}, if we further
assume that all $H_i$ are infinite cyclic groups then
 $\langle H_1,\dots, H_k\rangle$ is a free group with $k$ generators.
\end{corollary}


To make the analysis of the mapping class group simpler we will use a tropical degeneration of cluster mutations.
Namely, we consider the piecewise linear action of cluster mutations on the space of $g$-vectors.

\subsection{Mutation of $g$-vectors}\label{sec:g-vectors}

In this section we recall the definition of $g$-vectors.

Denote by  ${\mathbb T}_n$ the $n$-regular tree of clusters of the cluster algebra $\A$ of rank $n$, and let $t_0\in {\mathbb T}_n$. Denote by $B$ the exchange matrix at $t_0$.

\begin{proposition}[\cite{FZ4}, Proposition 3.13, Corollary 6.3]
Every pair $(B; t_0)$ gives rise
to a family of polynomials $F_{j;t} = F^{B;t_0}_{j;t}\in Z[u_1, . . . , u_n]$
and two families of integer
vectors 
$g_{j;t} = g^{B;t_0}_{j;t} = (g_{1j;t}, \ldots , g_{nj;t}) \in {\mathbb Z}^n$
(where $j \in\N$ and $t \in {\mathbb T}_n)$ with the following properties:
\begin{enumerate}
\item Each $F_{j;t}$ is not divisible by any $u_i$, and can be expressed as a ratio
of two polynomials in $u_1, \ldots , u_n$ with positive integer coefficients,
thus can be evaluated
in every semifield $\mathbb P$.
\item For any $j$ and $t$, we have
\begin{equation}
x_{j;t} = x^{g_{1j;t}}_1 \cdot\ldots\cdot x^{g_{nj;t}}_n
\frac{F_{j;t}|_\F(\hat y_1, \ldots , \hat y_n)}{F_{j;t}|_{\mathbb P} (y_1, \ldots , y_n)},
\end{equation}
where the elements $\hat y_j$ are given by
$\hat y_j = y_j \prod_i x^{b_{ij}}_i$.
\end{enumerate}
\end{proposition}
Here the tropical semifield ${\mathbb P}$ can be assumed to be trivial for our purposes, and $\F$ can be assumed to be a field of rational functions in $(x_1,\dots,x_n)$ with rational coefficients.

Mutations of $g$-vectors are described by the following conjecture~\cite{FZ4}.

\begin{conjecture}[\cite{FZ4}, Conjecture 7.12]
Let $t_0\longleftrightarrow t_1$
be two adjacent vertices in ${\mathbb T}^n$, and let $B^1 = \mu_k(B^0)$.
Then, for any $t \in {\mathbb T}^n$ and $a\in {\mathbb Z}^n_{\ge 0}$,
the $g$-vectors $g^{B^0;t_0}_{a;t} = (g_1, \ldots , g_n)$ and $g^{B^1;t_1}_{a;t} =
(g'_1, \ldots , g'_n )$ are related as follows:
\begin{equation}\label{eq:g-mutation}
g'_j =
\left\{
  \begin{array}{ll}
    g_k, & \hbox{if } j=k; \\
    g_j+[B^0_{jk}]_+ g_k-B^0_{jk} [g_k]_-, & \hbox{if } j\ne k,
  \end{array}
\right.
\end{equation}
where $X_+=max(X,0)$ and  $X_-=min(X,0)$ denote the positive and the negative part of the real number $X$.
\end{conjecture}
For skew-symmetric exchange matrices $B^0(B^1)$ the conjecture was proved in~\cite{DWZ2}.

In order to prove exponential growth we will use Equation~\ref{eq:g-mutation}. Moreover, we will also apply Equation~\ref{eq:g-mutation} to particular skew-symmetrizable exchange matrices. However, all the skew-symmetrizable exchange matrices we consider have skew-symmetric unfoldings, so Equation~\ref{eq:g-mutation} clearly holds.

We will consider all the exceptional types of cluster algebras one by one. Our aim is to find two sequences of mutations acting on the space $E_G$ of $g$-vectors as in Lemma~\ref{lem:ping-pong}.

\subsection{$X_6$ and $X_7$}
\label{sec:X6}

We start with cluster algebras with diagrams $X_6$ and $X_7$ shown in Fig.~\ref{allfig}. Let us label vertices of the diagram $X_6$ as shown in Fig.~\ref{fig:X6}.

\begin{figure}[!h]
\begin{center}
\psfrag{1}{\small $1$}
\psfrag{2}{\small $2$}
\psfrag{3}{\small $3$}
\psfrag{4}{\small $4$}
\psfrag{5}{\small $5$}
\psfrag{6}{\small $6$}
\psfrag{2_}{\scriptsize $4$}
\epsfig{file=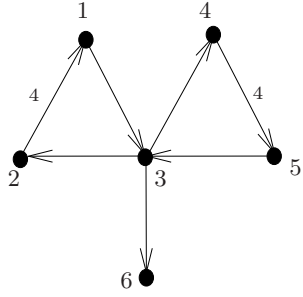,width=0.25\linewidth}
\caption{Diagram for  $X_6$}
\label{fig:X6}
\end{center}
\end{figure}

We consider two following mutation sequences
${\bf a}=[3, 2, 1]^{10}$ and ${\bf b}=[3, 5, 4, 2, 6]^4$ (by $[i_1,\dots,i_k]$ we mean a sequence of mutations $\mu_{i_k}\dots\mu_{i_1}$). By direct calculation we observe that both ${\bf a}$ and ${\bf b}$ preserve the diagram shown in Fig.~\ref{fig:X6}, or, in other words, they are elements of the mapping class group of $X_6$.

Note that both ${\bf a}$ and ${\bf b}$ act on the space $E_G$ of $g$-vectors of $X_6$ as described in Section~\ref{sec:g-vectors}.

Let us define following subsets of $E_G$:
\begin{eqnarray*}
X^+_{\bf a}(\varepsilon) &:= &\{(T-\nu, -T, 0,0,0,\nu), \text{ where } T>0, \nu>0, \nu<\varepsilon T\}, \\
X^-_{\bf a}(\varepsilon) &:= &\{(T, -T+\nu, -\nu,0,0,\nu), \text{ where } T>0, \nu>0, \nu<\varepsilon T\}, \\
X^+_{\bf b}(\varepsilon) &:= &\{(\nu, T-\nu, -T,0,0,T), \text{ where } T>0, \nu>0, \nu<\varepsilon T\}, \\
X^-_{\bf b}(\varepsilon) &:= &\{(-\nu, T-\nu, -T+\nu,0,0,T),\text{ where } T>0, \nu>0, \nu<\varepsilon T\}.
\end{eqnarray*}

We see by inspection for $\varepsilon <1/15$ that ${\bf a, a^{-1}, b, b^{-1}}$ act linearly on $X^\pm_{\bf a}$ ($X^\pm_{\bf b}$, correspondingly). In particular,
$$
\begin{array}{ll}
  {\bf a}(T-\nu, -T, 0,0,0,\nu) &= (T+14 \nu,-T-15\nu,0,0,0,\nu),  \\
  {\bf a^{-1}}(T, -T+\nu, -\nu,0,0,\nu) &= (T+15 \nu,-T-14\nu,-\nu,0,0,\nu), \\
  {\bf b}(\nu, T-\nu, -T,0,0,T) &= (-\nu,T+2\nu,-T-3\nu,0,0,T+3\nu), \\
  {\bf b^{-1}}(-\nu, T-\nu, -T+\nu,0,0,T) &= (-\nu,T+2\nu,-T-2\nu,0,0,T+3\nu).
\end{array}
$$
Note also that $X^+_{\bf a}(\varepsilon)$ is invariant under ${\bf a}$. Indeed,
$${\bf a}(T-\nu, -T, 0,0,0,\nu) = (T'-\nu,-T',0,0,0,\nu),$$
where $T'=T+15\nu$. Clearly, $\nu<\varepsilon T\le \varepsilon T'$.

\noindent
Similarly, $${\bf a^{-1}}(X^-_{\bf a}(\varepsilon))\subset X^-_{\bf a}(\varepsilon), \
{\bf b}(X^+_{\bf b}(\varepsilon))\subset X^+_{\bf b}(\varepsilon),\
{\bf b^{-1}} (X^-_{\bf b}(\varepsilon))\subset X^-_{\bf b}(\varepsilon).$$
Moreover, for any $v_z\in X^+_{\bf a}$ we have
$$\lim_{n\to\infty} \frac{{\bf a^n}(v_z)}{|{\bf a^n}(v_z)|}=
(\frac{1}{\sqrt{2}},-\frac{1}{\sqrt{2}},0,0,0,0),$$
and, similarly
$$\lim_{n\to\infty} \frac{{\bf a^{-n}}(v_z)}{|{\bf a^{-n}}(v_z)|}=
(\frac{1}{\sqrt{2}},-\frac{1}{\sqrt{2}},0,0,0,0,0)\ \text{for}\  v_z\in X^-_{\bf a},$$
$$\lim_{n\to\infty} \frac{{\bf  b^n}(v_z)}{|{\bf b^n}(v_z)|}=
(0,\frac{1}{\sqrt{3}},-\frac{1}{\sqrt{3}},0,0,\frac{1}{\sqrt{3}})\ \text{for}\ v_z\in X^+_{\bf b},$$
$$\lim_{n\to\infty} \frac{{\bf  b^{-n}}(v_z)}{|{\bf b^{-n}}(v_z)|}=
(0,\frac{1}{\sqrt{3}},-\frac{1}{\sqrt{3}},0,0,\frac{1}{\sqrt{3}})\ \text{for}\  v_z\in X^-_{\bf b}.$$

\noindent
Computations in Maple show that

\smallskip
\begin{center}
${\bf a}^{10}(0,1,-1,0,0,1)\in X^+_{\bf a}(\frac{1}{15})$, \qquad\quad
${\bf a}^{-10}(0,1,-1,0,0,1)\in X^-_{\bf a}(\frac{1}{15})$,\\
\smallskip
${\bf b}^{10}(1,-1,0,0,0,0)\in X^+_{\bf b}(\frac{1}{15})$, \qquad\quad
${\bf b}^{-10}(1,-1,0,0,0,0)\in X^-_{\bf b}(\frac{1}{15})$.
\end{center}

\smallskip

Since
${\bf a}^{\pm 10}$, ${\bf b}^{\pm 10}$ act continuously on $E_G$,
there is a sufficiently small $\epsilon>0$ and an integer $N>0$ such that
$${\bf a}^{N}\!\!\left(X^+_{\bf b}(\epsilon)\right)\!\subset\! X^+_{\bf a}(\epsilon),\quad 
{\bf a}^{-N}\!\!\left(X^+_{\bf b}(\epsilon)\right)\!\subset\! X^-_{\bf a}(\epsilon),\quad 
{\bf a}^{N}\!\!\left(X^-_{\bf b}(\epsilon)\right)\!\subset\! X^+_{\bf a}(\epsilon),\quad 
{\bf a}^{-N}\!\!\left(X^-_{\bf b}(\epsilon)\right)\!\subset\! X^-_{\bf a}(\epsilon).$$

Similarly,
$${\bf b}^{N}\!\!\left(X^+_{\bf a}(\epsilon)\right)\!\subset\! X^+_{\bf b}(\epsilon),\quad
{\bf b}^{-N}\!\!\left(X^+_{\bf a}(\epsilon)\right)\!\subset\! X^-_{\bf b}(\epsilon),\quad
{\bf b}^{N}\!\!\left(X^-_{\bf a}(\epsilon)\right)\!\subset\! X^+_{\bf b}(\epsilon),\quad
{\bf b}^{-N}\!\!\left(X^-_{\bf a}(\epsilon)\right)\!\subset\! X^-_{\bf b}(\epsilon).$$

Now define
$$X_{\bf a}=X^-_{\bf a}(\epsilon)\cup X^+_{\bf a}(\epsilon),\qquad
X_{\bf b}=X^-_{\bf b}(\epsilon)\cup X^+_{\bf b}(\epsilon)$$

Let $H_{\bf a}=\langle {\bf a}^N\rangle, \  H_{\bf b}=\langle {\bf b}^N\rangle$ be two infinite cyclic subgroups of mapping class group. One can easily see that the collection $H_{\bf a},\ H_{\bf b}$ and two subsets $X_{\bf a},\ X_{\bf  b}$ satisfy assumptions of Corollary~\ref{cor:ping-pong}.

Therefore, we obtain the following.

\begin{lemma}\label{lem:X6} The cluster algebra of type $X_6$ has exponential growth.
\end{lemma}

\begin{corollary} The cluster algebra of type $X_7$ has exponential growth.
\end{corollary}

\begin{proof} Indeed, the diagram $X_7$ contains the diagram $X_6$ as a subdiagram. Hence, exchange graph of $X_7$ contains exchange graph of $X_6$ as a subgraph, and therefore also grows exponentially.
\end{proof}

\subsection{$G_2^{(*,+)}$ and its unfolding $E_6^{(1,1)}$}\label{subsec:G_2*+}

There are two skew-symmetrizable matrices with diagram $G_2^{(*,+)}$. They are denoted by $G_2^{(1,3)}$ and $G_2^{(3,1)}$ according to Saito's notation for extended affine root systems~\cite{Sa}. These two matrices clearly define isomorphic cluster algebras. It was shown in~\cite{FST2} that both exchange matrices with diagram $G_2^{(*,+)}$ have an unfolding with diagram $E_6^{(1,1)}$. We will prove exponential growth of the cluster algebra with diagram $G_2^{(*,+)}$, and then deduce from it exponential growth of $E_6^{(1,1)}$.

The considerations are similar to those of Section~\ref{sec:X6}. Let us index the vertices of $G_2^{(*,+)}$ as shown in Fig.~\ref{fig:G_2*+}. The choice of labels $(3,1)$ and $(1,3)$ indicates which of the two matrices with this diagram we choose: the entries $\pm 3$ are located in the first and second columns and the third and fourth rows.


\begin{figure}[!h]
\begin{center}
\psfrag{1}{\small $1$}
\psfrag{2}{\small $2$}
\psfrag{3}{\small $3$}
\psfrag{4}{\small $4$}
\psfrag{3,1}{\scriptsize $3,1$}
\psfrag{1,3}{\scriptsize $1,3$}
\epsfig{file=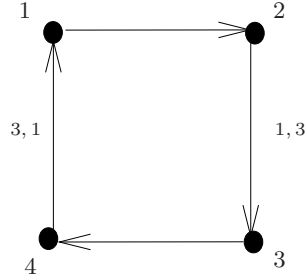,width=0.25\linewidth}
\caption{Diagram for $G_2^{(*,+)}$. The weight ``$m,n$'' on the arrow from vertex $i$ to vertex $j$ means that the ratio of $|b_{ij}|$ and $|b_{ji}|$ is equal to $m/n$. }
\label{fig:G_2*+}
\end{center}
\end{figure}

We will use the following  two mutation sequences: ${\bf a}=[1, 2, 3]^2$ and ${\bf b}=[2, 3, 4]^2$.
As above, both ${\bf a}$ and ${\bf b}$ are elements of the mapping class group of $G_2^{(*,+)}$, i.e. they preserve the diagram shown in  Fig.~\ref{fig:G_2*+}.
Note that ${\bf a},{\bf b}$ span a subgroup $\langle {\bf a},{\bf b}\rangle$  in the mapping class group of the diagram.

Consider the actions of ${\bf a}$ and ${\bf b}$ on the space $E_G$ of $g$-vectors. Endow $E_G$ with the standard dot product. For a vector $v\in E_G$ or subspace $V\subset E_G$ we denote their orthogonal complements by $v^\perp$ or $V^\perp$. Let $v_{\bf a}=(-1,-1,3,0)$. For sufficiently small $\epsilon$ we define cone
$$C_{\bf a}(\epsilon)=
\left\{\alpha v_{\bf a}+w, \text{ where } \alpha >0, w\in v_{\bf a}^\perp, \frac{|w|}{\alpha |v_{\bf a}|}< \epsilon\right\}$$

As we have mentioned above, the action of ${\bf a}$ on $E_G$ is piecewise linear. However, the action turns out to be linear on $C_{\bf a}(\epsilon)$ if $\epsilon$ is sufficiently small.

More precisely, ${\bf a}$ maps the $g$-vector  $v_z=v_{\bf a}+\bar z$ for sufficiently small $\bar z=(z_1,z_2,z_3,z_4)$
to the vector
\begin{equation}\label{eq:V4a}
{\bf a}(v_z)=v_{\bf a}+(2 z_1+2 z_2+z_3,z_1+3 z_2+z_3,-3 z_1-6 z_2-2 z_3,z_4).
\end{equation}

Direct computation shows that ${\bf a}=T^{-1}\left(
                                                \begin{array}{cccc}
                                                  1 & 1 & 0 & 0 \\
                                                  0 & 1 & 0 & 0 \\
                                                  0 & 0 & 1 & 0 \\
                                                  0 & 0 & 0 & 1 \\
                                                \end{array}
                                              \right)
 T$, where $T\in GL_4$. Note that the linear operator ${\bf a}$ contains (as a direct summand) the
Jordan  block with eigenvalue one
and corresponding eigenvector $v_{\bf a}$.


Then powers of ${\bf a}$ act on $C_{\bf a}(\epsilon)$ for small epsilon as follows.
Denote the second coordinate of the vector $T{\bar z}$ by
$\kappa_{\bf a}$, a simple computation shows $\kappa_{\bf a}=z_1+2 z_2+z_3$.

Then ${\bf a}^r(v_z)=v_z+ r\kappa_{
\bf a}(v_z) v_{\bf a}$.


Define $X^+_{\bf a}(\epsilon)=C_{\bf a}(\epsilon)\cap \{\kappa_{\bf a}>0\}$.
Then
$$\lim_{n\to \infty} \dfrac{{\bf a}^n v_z}{|{\bf a}^n v_z|}=\dfrac{v_{\bf a}}{|v_{\bf a}|} \ \text{if} \
v_z\in X^+_{\bf a}(\epsilon).$$

Similarly, define $X^-_{\bf a}(\epsilon)=C_{\bf a}(\epsilon)\cap \{\kappa_{\bf a}<0\}$. We have
$$\lim_{n\to \infty} \dfrac{{\bf a}^{-n} v_z}{|{\bf a}^{-n} v_z|}=\dfrac{v_{\bf a}}{|v_{\bf a}|}\ \text{for}\ v_z\in X^-_{\bf a}(\epsilon).$$

 From Equation~\ref{eq:V4a} we see that each $X^\pm_{\bf a}(\epsilon)$ is invariant under ${\bf a}^{\pm 1}$ for $\epsilon$ small enough.

Let $v_{\bf b}=(0,-1,1,1)$.
We consider the action of ${\bf b}$ on $E_G$ in a neighborhood of the ray
$\{\alpha\cdot v_{\bf b}\,\big| \alpha>0\}$.

Define the cone $C_{\bf b}(\epsilon)=
\left\{\alpha v_{\bf b}+w,\text{ where } \alpha >0, w\in v_{\bf b}^\perp, \frac{|w|}{\alpha |v_{\bf b}|}< \epsilon\right\}$.

For sufficiently small $\bar z=(z_1,z_2,z_3,z_4)$, ${\bf b}$ maps the $g$-vector $v_z=v_{\bf b}+\bar z$ to the
vector

\begin{equation}\label{eq:V4b}
{\bf b}(v_z)=v_{\bf b}+(z_1,4 z_2+2 z_3+z_4,-3 z_2-z_3-z_4,-3 z_2-2 z_3)
\end{equation}

The action is linear on $C_{\bf b}(\epsilon)$ and the corresponding linear operator is a direct sum of the identity operator on 2-dimensional space
and a $2\times 2$
Jordan  block with eigenvalue 1  with corresponding eigenvector $v_{\bf b}$.

Let $\kappa_{\bf b}(v_z)=z_2+(2/3) z_3+(1/3) z_4$, then we have ${\bf b}^r(v_z)=v_z+r \kappa_{\bf b}(v_z) v_{\bf b}$ for any positive integer $r$.
Denote $X_{\bf b}^+(\epsilon)=C_{\bf b}(\epsilon)\cap\{ \kappa_b>0\}$.
Then Equation~\ref{eq:V4b} implies that
$$\lim_{n\to \infty} \dfrac{{\bf b}^n v_z}{|{\bf b}^n v_z|}=\dfrac{v_{\bf b}}{|v_{\bf b}|} \ \text{for} \ v_z\in X^+_{\bf b}(\epsilon).$$

Similarly, we can define $X_{\bf b}^-(\epsilon)=C_{\bf b}(\epsilon)\cap\{ \kappa_{\bf b}<0\}$. Then
$$\lim_{n\to \infty} \dfrac{{\bf b}^{-n} v_z}{|{\bf b}^{-n} v_z|}=\dfrac{v_{\bf b}}{|v_{\bf b}|}\ 
\text{for} \ v_z\in X^-_{\bf b}(\epsilon).$$

One can also note that  $X^+_{\bf b}(\epsilon)$ is invariant under ${\bf b}$ and $X^-_{\bf b}(\epsilon)$ is invariant under ${\bf b}^{-1}$ for sufficiently small $\epsilon$.  Straightforward computations using Maple show that ${\bf b}^{\pm 10}(v_{\bf a})\in C_{\bf b}(\epsilon)$, where $\epsilon$ is small enough for Equation~\ref{eq:V4b} to hold.
Moreover, $\kappa_{\bf b}({\bf b}^{10}(v_{\bf a}))>0$, while $\kappa_{\bf b}({\bf b}^{-10}(v_{\bf a}))<0$.

Vice versa,
${\bf a}^{\pm 10}(v_{\bf b})\in C_{\bf a}(\epsilon)$, where $\epsilon$ is small enough for Equation~\ref{eq:V4a} to hold.
Also, $\kappa_{\bf a}({\bf a}^{10}(v_{\bf b}))>0$, while $\kappa_{\bf a}({\bf a}^{-10}(v_{\bf b}))<0$.

Hence, for any $\epsilon>0$ small enough we can find
a sufficiently large positive integer $N_\epsilon$
such that ${\bf b}^{N_\epsilon}(X_{\bf a}(\epsilon))\subset X^+_{\bf b}(\epsilon)$, ${\bf b}^{-N_\epsilon}(X_a(\epsilon))\subset X^-_{\bf b}(\epsilon)$,
${\bf a}^{N_\epsilon}(X_{\bf b}(\epsilon))\subset X^+_{\bf a}(\epsilon)$, ${\bf a}^{-N_\epsilon}(X_b(\epsilon))\subset X^-_{\bf a}(\epsilon)$.

Note that the collection of two infinite cyclic groups $H_{\bf a}=\langle{\bf a}^{N_\epsilon}\rangle$,
$H_{\bf b}=\langle {\bf b}^{N_\epsilon}\rangle$, and two sets $C_{\bf a}(\epsilon)$, $C_{\bf b}(\epsilon)$
satisfy the assumptions of Corollary~\ref{cor:ping-pong}.
Thus,

\begin{lemma}\label{lem:G_2*+} A cluster algebra with diagram of type $G_2^{(*,+)}$ has exponential growth.
\end{lemma}

\begin{corollary}\label{cor:E_6^11} The cluster algebra of type $E_6^{(1,1)}$ has exponential growth.
\end{corollary}
\begin{proof} The exchange matrix with diagram $E_6^{(1,1)}$ is an unfolding of the exchange matrix with diagram $G_2^{(*,+)}$.  In particular, any mutation in  $G_2^{(*,+)}$ is lifted to a sequence of mutations of $E_6^{(1,1)}$, and the mapping class group of $E_6^{(1,1)}$ contains the mapping class group of $G_2^{(*,+)}$ as a subgroup. Hence,  the growth of $E_6^{(1,1)}$ is exponential.
\end{proof}

\subsection{$G_2^{(*,*)}$ and its unfolding $E_8^{(1,1)}$}
\label{subsec:G_2**}

There are two distinct skew-symmetrizable matrices with diagram $G_2^{(*,*)}$, which are denoted by $G_2^{(3,3)}$ and $G_2^{(1,1)}$ (see~\cite[Table~6.3]{FST2}). We will prove that cluster algebras corresponding to matrices $G_2^{(1,1)}$ and $G_2^{(3,3)}$ have exponential growth.  The considerations are almost identical: one needs to take the same sequences of mutations, but different vectors $v_{\bf a},v_{\bf b},\kappa_{\bf a},\kappa_{\bf b}$. We will indicate below details that differ. Then exponential growth of the unfolding $E_8^{(1,1)}$ of $G_2^{(1,1)}$ follows.

The reasoning is similar to that from Section~\ref{subsec:G_2*+}. The following  diagram represents $G_2^{(1,1)}$, see Fig.~\ref{fig:G_2**}. The diagram of $G_2^{(3,3)}$ is obtained by reversing the orientation of all the arrows.


\begin{figure}[!h]
\begin{center}
\psfrag{4_}{\scriptsize $4$}
\psfrag{1}{\small $1$}
\psfrag{2}{\small $2$}
\psfrag{3}{\small $3$}
\psfrag{4}{\small $4$}
\psfrag{3,1}{\scriptsize $3,1$}
\psfrag{1,3}{\scriptsize $1,3$}
\epsfig{file=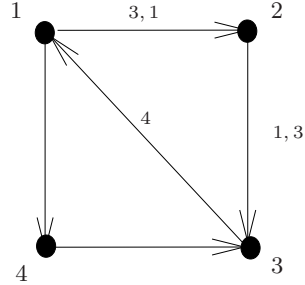,width=0.25\linewidth}
\caption{Diagram for $G_2^{(1,1)}$}
\label{fig:G_2**}
\end{center}
\end{figure}

Set
${\bf a}=[4, 1, 2]^4$, $v_{\bf a}=(-2,1,0,1)$ ($v_{\bf a}=(2,-3,0,-1)$ for $G_2^{(3,3)}$),
and ${\bf b}=[4,3,2]^4$, $v_{\bf b}=(0,-1,2,-1)$ ($v_{\bf b}=(0,3,-2,1)$ for $G_2^{(3,3)}$).

Define cones $C_{\bf a}(\epsilon)$ and $C_{\bf b}(\epsilon)$ as above.

 For small $\bar z=(z_1,z_2,z_3,z_4)$  define the $g$-vector $v_z=v_{\bf a}+\bar z$. Then
 \begin{equation}\label{eq:W4a}
{\bf a}(v_z)=v_{\bf a}+(13 z_1+18 z_2+6 z_4,-6 z_1-8 z_2-3 z_4,z_3,-6 z_1-9 z_2-2 z_4).
\end{equation}

For $G_2^{(3,3)}$ we have
 \begin{equation}\label{eq:W4aa}
{\bf a}(v_z)=v_{\bf a}+(13 z_1+6 z_2+6 z_4,-18 z_1-8 z_2-9 z_4,z_3,-6 z_1-3 z_2-2 z_4).
\end{equation}

As above, the action of ${\bf a}$ on $E_G$ is a linear transformation, which is a direct sum of the identity operator on 2-dimensional space and a $2\times 2$ Jordan  block with unit eigenvalue  and eigenvector $v_{\bf a}$.

Define
$\kappa_{\bf a}(v_z)=6z_1+9 z_2+3 z_4$ ($\kappa_{\bf a}(v_z)=6z_1+3 z_2+3 z_4$ for $G_2^{(3,3)}$).  Then for any positive integer $r$ we have ${\bf a}^r(v_z)=v_z-r\kappa_{\bf a}(v_z) v_{\bf a}$ (${\bf a}^r(v_z)=v_z+r\kappa_{\bf a}(v_z) v_{\bf a}$ in the case of $G_2^{(3,3)}$).

Define $X^+_{\bf a}(\epsilon)$, $X^-_{\bf a}(\epsilon)$, $X^+_{\bf b}(\epsilon)$, and $X^-_{\bf b}(\epsilon)$ as above.

 From Equation~\ref{eq:W4a} (~\ref{eq:W4aa}) we see that each $X^\mp_{\bf a}(\epsilon)$ is invariant under ${\bf a^{\pm 1}}$ for $\epsilon$ small enough. If $v_z\in X^-_{\bf a}(\epsilon)$ then
$\lim_{n\to \infty} \dfrac{{\bf a}^n v_z}{|{\bf a}^n v_z|}=\dfrac{v_{\bf a}}{|v_{\bf a}|}$.
 If $v_z\in X^+_{\bf a}(\epsilon)$ then
$\lim_{n\to \infty} \dfrac{{\bf a}^{-n} v_z}{|{\bf a}^{-n} v_z|}=\dfrac{v_{\bf a}}{|v_{\bf a}|}$.

 For sufficiently small $\bar z=(z_1,z_2,z_3,z_4)$ define the $g$-vector $v_z=v_{\bf b}+\bar z$.
Then
\begin{equation}\label{eq:W4b}
{\bf b}(v_z)=v_{\bf b}+(z_1,-8 z_2-6 z_3-3 z_4,18 z_2+13 z_3+6 z_4,-9 z_2-6 z_3-2 z_4)
\end{equation}
For $G_2^{(3,3)}$ we have
\begin{equation}\label{eq:W4bb}
{\bf b}(v_z)=v_{\bf b}+(z_1,-8 z_2-18 z_3-9 z_4,6 z_2+13 z_3+6 z_4,-3 z_2-6 z_3-2 z_4)
\end{equation}

The corresponding linear transformation is a direct sum of an identity operator on 2-dimensional space and a Jordan  block of size $2\times 2$ with eigenvalue one
and the corresponding eigenvector $v_{\bf b}$.
If
$\kappa_{\bf b}=9z_2+6 z_3+3 z_4$ ($\kappa_{\bf b}=3 z_2+6 z_3+3 z_4$ for $G_2^{(3,3)}$) then ${\bf b}^r(v_z)=v_z+ r\kappa_{\bf b}(v_z) v_{\bf b}$ (${\bf b}^r(v_z)=v_z- r\kappa_{\bf b}(v_z) v_{\bf b}$ for $G_2^{(3,3)}$).

Note now, that ${\bf b}(X_{\bf b}^+(\epsilon))\subset X_{\bf b}^+(\epsilon)$, and ${\bf b^{-1}}(X_{\bf b}^-(\epsilon))\subset X_{\bf b}^-(\epsilon)$. Furthermore, as in the previous case, we have
\begin{center}
$\lim_{n\to \infty} \dfrac{{\bf b}^n v_z}{|{\bf b}^n v_z|}=\dfrac{v_{\bf b}}{|v_{\bf b}|}$ for $v_z\in X^+_{\bf b}(\epsilon)$,\\
$\lim_{n\to \infty} \dfrac{{\bf b}^{-n} v_z}{|{\bf b}^{-n} v_z|}=\dfrac{v_{\bf b}}{|v_{\bf b}|}$ for $v_z\in X^-_{\bf b}(\epsilon)$.
\end{center}
Again, ${\bf b}^{\pm 10}(v_{\bf a})\in X^\pm_{\bf b}(\epsilon)$, ${\bf a}^{\pm 10}(v_{\bf b})\in X^\mp_{\bf a}(\epsilon)$, where $\epsilon$ is small enough for Equations~\ref{eq:W4a} and~\ref{eq:W4b}  to hold.

Therefore, we can conclude that for any  $\epsilon>0$ small enough we can find
a sufficiently large positive integer $N_\epsilon$
such that ${\bf b}^{N_\epsilon}(X_{\bf a}(\epsilon))\subset X^+_{\bf b}(\epsilon)$, ${\bf b}^{-N_\epsilon}(X_a(\epsilon))\subset X^-_{\bf b}(\epsilon)$,
${\bf a}^{N_\epsilon}(X_{\bf b}(\epsilon))\subset X^-_{\bf a}(\epsilon)$, ${\bf a}^{-N_\epsilon}(X_b(\epsilon))\subset X^+_{\bf a}(\epsilon)$.

Now we apply Corollary~\ref{cor:ping-pong}  to get the result.

\begin{lemma}\label{lem:G_233} Cluster algebras of type $G_2^{(1,1)}$ and $G_2^{(3,3)}$  have exponential growth.
\end{lemma}

Equivalently,

\begin{corollary}\label{lem:G_2**} Cluster algebras with diagram of type $G_2^{(*,*)}$ have exponential growth.
\end{corollary}

The fact that $E_8^{(1,1)}$ is an unfolding of $G_2^{(1,1)}$ implies the following corollary.

\begin{corollary}\label{cor:E_8^11} The cluster algebra of type $E_8^{(1,1)}$ has exponential growth.
\end{corollary}

\subsection{$F_4^{(*,+)}$, its unfolding $E_7^{(1,1)}$, and $F_4^{(*,*)}$}

In this section we will prove exponential growth for $F_4^{(*,+)}$, its unfolding $E_7^{(1,1)}$, and $F_4^{(*,*)}$.
The arguments follow almost literally the arguments of Sections~\ref{subsec:G_2*+} and~\ref{subsec:G_2**}.
Therefore we describe below only the differences between the cases in question and the cases $G_2^{(*,*)}$, $G_2^{(*,+)}$.

The  diagram representing $F_4^{(*,+)}$ is shown in Fig.~\ref{fig:F_4*+} (again, labels show the choice of one of the two matrices, which differ by permutations of rows and columns only).


\begin{figure}[!h]
\begin{center}
\psfrag{1}{\small 1}
\psfrag{2}{\small 2}
\psfrag{3}{\small 3}
\psfrag{4}{\small 4}
\psfrag{5}{\small 5}
\psfrag{6}{\small 6}
\psfrag{1,2}{\scriptsize 1,2}
\psfrag{2,1}{\scriptsize 2,1}
\epsfig{file=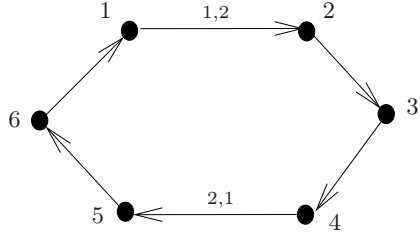,width=0.35\linewidth}
\caption{Diagram for $F_4^{(*,+)}$}
\label{fig:F_4*+}
\end{center}
\end{figure}

\noindent
Let
${\bf a}=[5,4,3,2, 1]$
and ${\bf b}=[2,1,6,5,4]$,
$v_{\bf a}=(-1,0,0,0,1,0)$, $v_{\bf b}=(0,1,0,-1,0,0)$.

In a neighborhood of $v_{\bf a}$ set $v_z=v_{\bf a}+(z_1,z_2,z_3,z_4,z_5,z_6)$. Then
\begin{equation}\label{eq:Y6a}
{\bf a}(v_z)=v_{\bf a}+
(-z_1-z_2-z_3-z_4-2 z_5,2 z_1+z_2+z_3+z_4+2 z_5,z_2,z_3,z_4+z_5,z_6)
\end{equation}
for sufficiently small $(z_1,z_2,z_3,z_4,z_5,z_6)$.

 Note that the Jordan form of the linear operator ${\bf a}$ is a direct sum of an identity operator, negative one times an identity operator, a rotation operator of order four (with eigenvalues of magnitude one) and a $2\times 2$
Jordan  block with eigenvalue one
whose eigenvector is $v_{\bf a}$.

Computing coordinates of the corresponding transformation matrix we set
$\kappa_{\bf a}=\frac{1}{2}(z_1+z_2+z_3+z_4+z_5)$. Then ${\bf a}^r(v_z)=v_z+ r\kappa_{
\bf a}(v_z) v_{\bf a}$ whenever $r$ is a multiple of four.

In a neighborhood of $v_{\bf b}$ we denote $v_z=v_{\bf b}+(z_1,z_2,z_3,z_4,z_5,z_6)$. Then
\begin{equation}\label{eq:Y6b}
{\bf b}(v_z)=v_z+
(z_6,2 z_1+z_2,z_3,-2 z_1-2 z_2-z_4-2 z_5-2 z_6,z_1+z_2+z_4+z_5+z_6,z_5)
\end{equation}
for sufficiently small $(z_1,z_2,z_3,z_4,z_5,z_6)$.

Similarly, the linear transformation ${\bf b}$ is a direct sum of the $2\times 2$ Jordan  block with eigenvalue one
and an identity operator, negative one times an identity operator, and a
rotation of order four. The eigenvector corresponding to the Jordan block is $v_{\bf b}$.

Set $\kappa_{\bf b}= z_1+ (1/2) z_2+ (1/2) z_4+ z_5+ z_6$. Then ${\bf b}^r(v_z)=v_z+ r \kappa_{\bf b}(v_z) v_{\bf b}$ whenever $r$ is a multiple of four.

As above, using Corollary~\ref{cor:ping-pong} we conclude:

\begin{lemma}\label{lem:F_4*+}
The cluster algebra of type $F_4^{(*,+)}$ has exponential growth.
\end{lemma}

\begin{corollary}\label{cor:E_7^11} The cluster algebra of type $E_7^{(1,1)}$ has exponential growth.
\end{corollary}
\begin{proof}  $E_7^{(1,1)}$ is an unfolding of $F_4^{(*,+)}$.
\end{proof}


Finally, we show that the growth of cluster algebras with diagram $F_4^{(*,*)}$ is exponential. As in the case of $G_2^{(*,*)}$, there are two distinct skew-symmetrizable matrices with this diagram, which correspond to extended affine root systems  $F_4^{(1,1)}$ and $F_4^{(2,2)}$ (see~\cite{FST2}). Below we consider $F_4^{(1,1)}$. The considerations for $F_4^{(2,2)}$ are almost identical; we give the differing details in parentheses.

  The  diagram representing $F_4^{(1,1)}$ is shown in Fig.~\ref{fig:F_4**}. The diagram representing $F_4^{(2,2)}$ is obtained by reversing the orientations of all the arrows.


\begin{figure}[!h]
\begin{center}
\psfrag{1}{\small 1}
\psfrag{2}{\small 2}
\psfrag{3}{\small 3}
\psfrag{4}{\small 4}
\psfrag{5}{\small 5}
\psfrag{6}{\small 6}
\psfrag{1,2}{\scriptsize 1,2}
\psfrag{2,1}{\scriptsize 2,1}
\epsfig{file=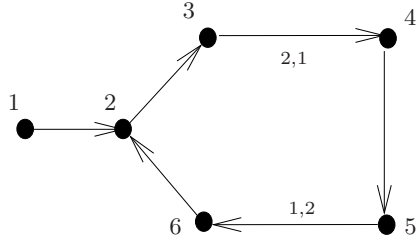,width=0.35\linewidth}
\caption{Diagram for  $F_4^{(1,1)}$}
\label{fig:F_4**}
\end{center}
\end{figure}

In both cases we considered the same pair of elements of the mapping class group
${\bf a}=[1,2,3,4,5]^2$ and ${\bf b}=[4,5,6,1,2]^2$.

Then choose
$v_{\bf a}=(-1,-1,-1,1,1,0)$, $v_{\bf b}=(-1/2,1,0,-1/2,-1/2,1/2)$ for $F_4^{(1,1)}$
(for $F_4^{(2,2)}$ \  $v_{\bf a}=(1,1,1,-2,-2,0)$, $v_{\bf b}=(1,-2,0,2,2,-1)$).

In a neighborhood of $v_{\bf a}$, denote $v_z=v_{\bf a}+(z_1,z_2,z_3,z_4,z_5,z_6)$. Then
\begin{equation}\label{eq:Z6a}
{\bf a}(v_z)= v_{\bf a}+
(-z_3-2 z_4,-z_1-z_2-z_3-2 z_4-2 z_5,z_1,z_2+z_3+2 z_4+z_5,z_3+z_4+z_5,z_6)
\end{equation}
for sufficiently small $(z_1,z_2,z_3,z_4,z_5,z_6)$ (for $F_4^{(1,1)}$).

For $F_4^{(2,2)}$ we have
\begin{multline}\label{eq:Z6aa}
{\bf a}(v_z)=v_{\bf a}+ \\ +
(z_3,z_1+z_2+z_3+z_4,z_1+2 z_2+2 z_3+z_4+z_5,-2 z_1-2 z_2-2 z_3-z_4-z_5,-2 z_2-2 z_3-z_4,z_6).
\end{multline}

 The linear operator ${\bf a}$ is a direct sum of a
Jordan  block of size 2 with eigenvalue one and  eigenvector $v_{\bf a}$,
an identity operator, and a rotation of order three.

Define
$\kappa_{\bf a}=\frac{1}{3}z_1+\frac{2}{3} z_2+ z_3+\frac{4}{3} z_4+\frac{2}{3} z_5$ ($\kappa_{\bf a}=\frac{1}{3}z_1+\frac{2}{3} z_2+ z_3+\frac{2}{3} z_4+ \frac{1}{3}z_5$ for $F_4^{(2,2)}$), then ${\bf a}^r(v_z)=v_z+r\kappa_{\bf a}(v_z) v_{\bf a}$
whenever $r$ is a multiple of three.

In a neighborhood of $v_{\bf b}$, denote $v_z=v_{\bf b}+(z_1,z_2,z_3,z_4,z_5,z_6)$. Then
\begin{multline}\label{eq:Z6b}
{\bf b}(v_z)=v_{\bf b}+\\ +
(2 z_5+z_6,-z_2-2 z_4-4 z_5-2 z_6,z_3,z_4+z_5+z_6,z_1+z_2+z_4+2 z_5+z_6,-z_1)
\end{multline}
for sufficiently small $(z_1,z_2,z_3,z_4,z_5,z_6)$.

For $F_4^{(2,2)}$ we have
\begin{multline}\label{eq:Z6bb}
{\bf b}(v_z)= v_{\bf b}+\\ +
(z_5+z_6,-z_2-z_4-2 z_5-2 z_6,z_3,z_4+z_5+2 z_6,2 z_1+2 z_2+z_4+2z_5+2z_6,-z_1).
\end{multline}

 The linear operator ${\bf b}$ is a direct sum of an identity operator, a rotation of order three and  a $2\times 2$
Jordan  block with eigenvalue one and
 eigenvector $v_{\bf b}$.

Define $\kappa_{\bf b}=\frac{2}{3}z_1+ \frac{4}{3} z_2+ \frac{4}{3} z_4+ \frac{8}{3} z_5+ 2 z_6$ ($\kappa_{\bf b}=\frac{1}{3} z_1+ \frac{2}{3} z_2+ \frac{1}{3} z_4+ \frac{2}{3} z_5+ z_6$ for $F_4^{(2,2)}$).
 Then ${\bf b}^r(v_z)=v_z- r\kappa_{\bf b}(v_z) v_{\bf b}$ whenever $r$ is a
multiple of three.

As above, using Corollary~\ref{cor:ping-pong} we conclude

\begin{lemma}\label{lem:F_411}
Cluster algebras of type $F_4^{(1,1)}$ and $F_4^{(2,2)}$ have exponential growth.
\end{lemma}

Equivalently,

\begin{corollary}\label{lem:F_4**}
Cluster algebras with diagram of type $F_4^{(*,*)}$ have exponential growth.
\end{corollary}

\section{Growth rates of affine cluster algebras}\label{sec:Hugh}
We are left with exceptional cluster algebras of affine type. This section is devoted to the proof of linear growth of affine cluster algebras. We start with skew-symmetric (simply-laced) affine cluster algebras (whose diagrams can be understood as quivers), and then use unfoldings to complete the proof of Theorem~\ref{thm-main} in the coefficient-free case.

\smallskip

Let $Q$ be a quiver without oriented cycles, and with $n$ vertices.
Let $A_Q$ be the cluster algebra associated to $Q$.

Let $k$ be an algebraically
closed field.  We write $kQ$ for the path algebra of $Q$, and
\kqm for its module category.  The bounded derived category of this
abelian category is denoted $D^b(kQ)$.  This category is triangulated, and
therefore
equipped
with a shift autoequivalence $[1]$; it also has an Auslander-Reiten
autoequivalence $\tau$.

Given a triangulated category and an
auto-equivalence, there is an orbit category, in which objects
in the same orbit with respect to the autoequivalence are isomorphic.
By definition,
the cluster category is the orbit category
$\mathcal C_Q=D^b(kQ)/[1]\tau^{-1}$, which is again triangulated by a result
of Keller \cite{ke}.  This category is called the cluster category associated
to $Q$, and was introduced in \cite{BMRRT}.

Thanks to the embedding of objects of \kqm as stalk complexes in degree
zero
inside $D^b(kQ)$,
there is a functor from \kqm to $\mathcal C_Q$, which embeds
\kqm as a (non-full) subcategory of $\mathcal C_Q$.
We write $P_i$ for the indecomposable projective $kQ$ module with
simple top at vertex $i$; we also write $P_i$ for the corresponding
object of $\mathcal C_Q$ via the above embedding.



An object $E$ in $\mathcal C_Q$ is called {rigid} if
it satisfies $\Ext^1_{\mathcal C_Q}(E,E)=0$.
The crucial result relating the cluster algebra to the cluster category
is the following \cite{CK,BMRTCK}:
the cluster variables in the cluster
algebra $A_Q$ are naturally in one-one correspondence with
the rigid indecomposables of $\mathcal C_Q$.  We will make the (slightly
non-standard) choice of identifying the cluster variable $u_i$ from
the initial seed with $P_i$.
A (basic) cluster tilting object in $\mathcal C_Q$ is the direct sum of
the collection of rigid indecomposables objects corresponding to the
cluster variables of some cluster.

Say that
two rigid indecomposable objects $E, F$ in $\mathcal C_Q$
are compatible if
$\Ext^1_{\mathcal C_Q}(E,F)=0$.  (By the 2-Calabi-Yau property of cluster
categories, this is equivalent to the condition that
$\Ext^1_{\mathcal C_Q}(F,E)=0$.)
Two rigid indecomposables are compatible if and only if the corresponding
cluster variables are both contained in some cluster.
Cluster tilting objects
can be given a representation-theoretic
description: $T$ is a cluster tilting object if $T$ is the direct sum
of a maximal collection of pairwise-compatible distinct rigid indecomposable
objects in $\mathcal C_Q$.

The autoequivalence $\tau$ of $D^b(kQ)$ descends to an autoequivalence of
$\mathcal C_Q$.  It therefore induces an action on the indecomposable
objects of $C_Q$.
Write $X_{i}^p$ for the indecomposable object $\tau^p P_i$, where $p\in
\mathbb Z$ and $1\leq i \leq n$.  These objects are pairwise non-isomorphic,
and each
is rigid.  We refer to these indecomposables as
{\it transjective}.  It will
be convenient to define a function $q$ on the transjective indecomposable
modules by setting $q(X^p_i)=p$.

There are also a finite number of other rigid indecomposable objects
in $\mathcal C_Q$.
They are referred to as the regular rigid indecomposable objects.
They lie in finite $\tau$-orbits.

We now prove a sequence of lemmas:

\begin{lemma}\label{two} Any cluster tilting object contains at least two
(non-isomorphic) indecomposable transjective summands.\end{lemma}

\begin{proof} We first
show that no cluster tilting object has exactly one transjective summand.
Suppose that $X\oplus R$ were a cluster tilting object, with $X$
indecomposable tranjective, and $R$ regular.  Since all the regular
indecomposable summands
of $R$ lie in finite $\tau$-orbits, there is some non-zero
$m\in \mathbb Z$ such that
 $\tau^m R \simeq R$.  Since $\tau$ is an auto-equivalence of $\mathcal C_Q$,
it follows that $\tau^{tm}P$ is compatible with $\tau^{tm}R\simeq R$ for
any $t\in\mathbb Z$.  This would mean that the cluster algebra
$A_Q$ has a collection of
$n-1$ cluster variables contained in an infinite number of clusters,
which is impossible.  (It is always the case that
$n-1$ cluster variables are contained in either
$0$ or $2$
clusters.)

It now follows that no cluster tilting object in $\mathcal C_Q$  contains zero
tranjective summands either, since any cluster tilting object can be obtained by
a finite number of mutations from the cluster tilting object $\bigoplus_i P_i$
\cite{BMRRT}.
Since each mutation changes exactly one summand of the cluster
tilting object, a sequence of mutations leading to a cluster tilting object
with no transjective summands would have to pass through a
cluster tilting object with exactly one transjective summand, which we have already shown is impossible. This proves the lemma.
\end{proof}

\begin{lemma} \label{one}
There is a bound $N$ such that
any tranjective rigid indecomposable compatible with $X_i^p$ is of
the form $X_j^r$ with for some $1\leq j\leq n$ and
$p-N\leq r\leq p+N$.  \end{lemma}

\begin{proof}  Since $\tau$ is an autoequivalence, it suffices to check
the statement for one element in each transjective $\tau$-orbit.
We will check it for each $P_i$.

Fix $i$ with $1\leq i \leq n$.
Consider the cluster algebra $A_i$ obtained by freezing the vertex $i$.  The
principal part of the exchange matrix of $A_i$ corresponds to
the quiver $Q$ with the
vertex $i$ removed.  This is a collection of Dynkin quivers, and thus
corresponds to a cluster algebra of finite type.  It follows that there
are only finitely many cluster variables in $A_i$.  These cluster
variables correspond to the cluster variables of $A_Q$ which are compatible
with $P_i$.  Since there are only finitely many of them, we can pick
a bound $N_i$ so that all the indecomposable transjective objects
compatible with $P_i$ are of the
form $X_j^r$ with $-N_i\leq r \leq N_i$.  Now set $N$ to be the maximum
of all the $N_i$.
\end{proof}

Let $T$ be a cluster tilting object.  Take the mean value of $q(E)$ as
$E$ runs through the transjective indecomposable
summands of $T$ (a non-empty set by
Lemma \ref{two}), and denote that mean value by $q(T)$.

\begin{corollary} \label{cor} If $M$ is a transjective summand of a cluster
tilting object $T$, then
$|q(T)-q(M)|\leq N$.  \end{corollary}

\begin{lemma} \label{three} If $T$ and $T'$ are cluster tilting objects
related by a single mutation, then
$|q(T')-q(T)|\leq N$.  \end{lemma}

\begin{proof}
Let $M$ be the summand of $T$ which does not appear in $T'$, and let
$M'$ be the summand of $T'$ which does not appear in $T$.  If neither
$M$ not $M'$ is transjective, then $q(T')=q(T)$, and we are done.
Otherwise, without loss of generality, suppose that $M$ is transjective.

If $E$ is any other transjective summand of $T$ (and there is at least one
such $E$ by Lemma \ref{two}), then $|q(E)-q(M)|\leq N$ by Lemma
\ref{one}.  This implies the desired result if $M'$ is not transjective.

If $M'$ is transjective, $|q(M')-q(E)|\leq N$ by Lemma \ref{one} again.  It
follows that the difference between the sum of the $q$-values of $T$
and $T'$ is at most $2N$, and thus their mean values differ by at most $N$.
\end{proof}

\begin{theorem} The growth rate of any affine simply-laced cluster algebra
is linear.\end{theorem}

\begin{proof}  Pick a starting cluster $T$.  Let $M$ be a transjective
summand of $T$.
Let $T'$ be obtained by applying
$k$ mutations to $T$.  Let $M'$ be any transjective summand of $T'$.
Applying Corollary \ref{cor} twice and Lemma \ref{three} once, it follows
that $|q(M')-q(M)|\leq N(k+2)$.  The number of transjective indecomposable
objects within this range is $2nN(k+2)$, while the number of
regular rigid indecomposable objects is finite.  It follows that the
number of cluster variables which can be obtained by $k$ mutations
starting from a given cluster is linearly bounded in $k$, as desired.
\end{proof}

\begin{corollary}
The growth rate of any affine cluster algebra is linear.
\end{corollary}
\begin{proof}
The diagrams $\t B_n$ and $\t C_n$ are s-decomposable, therefore, the vertices of the exchange graph of any cluster algebra with one of these diagrams are indexed by the triangulations of the corresponding orbifold (depending only on the diagram). This implies that the growth is the same for any skew-symmetrizable matrix with diagram $\t B_n$ (or $\t C_n$). Further, for any of these diagrams there is a matrix with an affine unfolding ($\t D_{n+1}$ for $\t B_n$, and $\t D_{n+2}$ for $\t C_n$). Since the growth rate of any cluster algebra is not faster than the growth rate of any its unfolding, we obtain linear growth for any cluster algebra with diagram $\t B_n$ or $\t C_n$.

For either of the diagrams $\t F_4$ and $\t G_2$ there are two skew-symmetrizable matrices with these diagrams, see~\cite[Table~6.3]{FST2}. All four matrices have affine unfoldings, namely, $\t E_6$ and $\t E_7$ for $\t F_4$, and $\t D_4$, $\t E_6$ for $\t G_2$. Again, this implies linear growth.
\end{proof}

The latter Corollary accomplishes the proof of Theorem~\ref{thm-main} in coefficient-free case.

\section{Coefficients}
\label{coeff}

In this section we prove the following lemma.

\begin{lemma}
\label{coefficients}
The growth rate of a cluster algebra does not depend on its coefficients.

\end{lemma}

\begin{proof}
It is easy to see from the definition that the exchange graph of a cluster algebra covers the exchange graph of the coefficient-free cluster algebra with the same exchange matrix. In particular, we have nothing to prove for algebras with exponential growth, so we only need to explore cases (1)--(4) from Theorem~\ref{thm-main}.

In~\cite{FZ3}, Fomin and Zelevinsky conjectured~\cite[Conjecture 4.14]{FZ3} that the exchange graph of a cluster algebra depends only on the exchange matrix. This conjecture is known to be true in many cases, including:
\begin{itemize}
\item
for cluster algebras of finite type~\cite{FZ2}, which covers case (2a);
\item
for cluster algebras of rank 2 (immediately following from the finite type case), which covers case (1);
\item
for cluster algebras from surfaces~\cite{FT} and orbifolds~\cite{FST3}, which covers cases (2b) and (3).
\item
for skew-symmetric cluster algebras~\cite{CKLP}, which covers case (4a) (and
parts of the previous cases).
\end{itemize}

Further, the unfolding argument does not depend on coefficients, so case (4b) is implied by (4a).
\end{proof}

This completes the proof of Lemma~\ref{coefficients} and thus also the proof of Theorem~\ref{thm-main}.

\bibliographystyle{amsalpha}

\end{document}